\documentclass[10pt]{article}

\usepackage{amsmath}
\usepackage{amssymb}

\usepackage{graphicx}

\usepackage{cite}

\usepackage{color}
\usepackage{enumerate}
\usepackage{relsize}
\usepackage{amsthm}

\usepackage[labelfont=bf,labelsep=period,justification=raggedright]{caption}

\makeatletter
\renewcommand{\@biblabel}[1]{\quad#1.}
\makeatother

\date{}

\theoremstyle{plain}
\newtheorem{thm}{Theorem}
  \theoremstyle{plain}
  \newtheorem{lem}[thm]{Lemma}
  \theoremstyle{definition}
  
  \theoremstyle{plain}
  
  \theoremstyle{remark}
  
\def\NN{{\mathbb{N}}}

\def\RR{{\mathbb{R}}}

\def\D{{\mathcal D}}

\def\Esp{{\mathbb{E}}}
\def\F{{\mathcal F}}

\def\Prob{{\mathbb{P}}}

\begin{document}

\begin{flushleft}
{\Large 
\textbf{Estimation of Distribution Overlap of Urn Models}
}
\\
Jerrad Hampton$^{1}$, 
Manuel E. Lladser$^{1,\ast}$
\\
\bf{1} Department of Applied Mathematics, University of Colorado, Boulder, Colorado, United States of America
\\
$\ast$ E-mail: manuel.lladser@colorado.edu\end{flushleft}

\section*{Abstract}
A classical problem in statistics is estimating the expected coverage of a sample, which has had applications in gene expression, microbial ecology, optimization, and even numismatics. Here we consider a related extension of this problem to random samples of two discrete distributions. Specifically, we estimate what we call the dissimilarity probability of a sample, i.e., the probability of a draw from one distribution not being observed in $k$ draws from another distribution. We show our estimator of dissimilarity to be a $U$-statistic and a uniformly minimum variance unbiased estimator of dissimilarity over the largest appropriate range of $k$. Furthermore, despite the non-Markovian nature of our estimator when applied sequentially over $k$, we show it converges uniformly in probability to the dissimilarity parameter, and we present criteria when it is approximately normally distributed and admits a consistent jackknife estimator of its variance. As proof of concept, we analyze V35 16S rRNA data to discern between various microbial environments. Other potential applications concern any situation where dissimilarity of two discrete distributions may be of interest. For instance, in SELEX experiments, each urn could represent a random RNA pool and each draw a possible solution to a particular binding site problem over that pool. The dissimilarity of these pools is then related to the probability of finding binding site solutions in one pool that are absent in the other.


\section*{Introduction}
An inescapable problem in microbial ecology is that a sample from an environment typically does not observe all species present in that environment. In~\cite{LladserGouetReeder}, this problem has been recently linked to the concepts of \emph{coverage probability} (i.e. the probability that a member from the environment is represented in the sample) and the closely related \emph{discovery} or \emph{unobserved probability} (i.e. the probability that a previously unobserved species is seen with another random observation from that environment). The mathematical treatment of coverage is not limited, however, to microbial ecology and has found applications in varied contexts, including gene expression, microbial ecology, optimization, and even numismatics.

The point estimation of coverage and discovery probability seem to have been first addressed by Turing and Good~\cite{Good} to help decipher the Enigma Code, and subsequent work has provided point predictors and prediction intervals for these quantities under various assumptions~\cite{Esty,Mao,LijoiEtAl,LladserGouetReeder}.

Following Robbins~\cite{Robbins} and in more generality Starr~\cite{Starr}, an unbiased estimator of the expected discovery probability of a sample of size $n$ is
\begin{equation}
\label{ide:Starr}
\sum_{k=1}^r\frac{{r-1\choose k-1}}{{n+r\choose k}}\cdot N(k,n+r),
\end{equation}
where $N(k,n+r)$ is the number of species observed exactly $k$-times in a sample with replacement of size $(n+r)$. Using the theory of U-statistics developed by Halmos~\cite{Halmos}, Clayton and Frees~\cite{ClaytonFrees} show that the above estimator is the \emph{uniformly minimum variance unbiased estimator} (UMVUE) of the expected discovery probability of a sample of size $n$ based on an enlarged sample of size $(n+r)$.

A quantity analogous to the discovery probability of a sample from a single environment but in the context of two environments is \emph{dissimilarity}, which we broadly define as the probability that a draw in one environment is not represented in a random sample (of a given size) from a possibly different environment. Estimating the dissimilarity of two microbial environments is therefore closely related to the problem of assessing the species that are unique to each environment, and the concept of dissimilarity may find applications to measure sample quality and allocate additional sampling resources, for example, for a more robust and reliable estimation of the UniFrac distance~\cite{LozuponeKnight,LozuponeEtAl} between pairs of environments. Dissimilarity may find applications in other and very different contexts. For instance, in SELEX experiments~\cite{TuerkGold}---a laboratory technique in which an initial pool of synthesized random RNA sequences is repeatedly screened to yield a pool containing only sequences with given biological functions---the dissimilarity of two RNA pools corresponds to the probability of finding binding site solutions in one pool that are absent in the other.

In this manuscript, we study an estimator of dissimilarity probability similar to Robbins' and Starr's statistic for discovery probability. Our estimator is optimal among the appropriate class of unbiased statistics, while being approximately normally distributed in a general case. The variance of this statistic is estimated using a consistent jackknife. As proof of concept, we analyze samples of processed V35 16S rRNA data from the Human Microbiome Project~\cite{HMPAnalysisPaper}.

\subsection*{Probabilistic Formulation and Inference Problem}

To study dissimilarity probability, we use the mathematical model of a pair of urns, where each urn has an unknown composition of balls of different colors, and where there is no a priori knowledge of the contents of either urn. Information concerning the urn composition is inferred from repeated draws with replacement from that urn.

In what follows, $X_1,X_2,\ldots$ and $Y_1,Y_2,\ldots$ are independent sequences of independent and identically distributed (i.i.d.) discrete random variables with probability mass functions $\Prob_x$ and $\Prob_y$, respectively. Without loss of generality we assume that $\Prob_x$ and $\Prob_y$ are supported over possibly infinite subsets of $\NN=\{1,2,3,\ldots\}$, and think of outcomes from these distributions as ``colors'': i.e. we speak of color-$1$, color-$2$, etc. Let $I_x$ denote the set of colors $i$ such that $\Prob_x(i)>0$, and similarly define $I_y$. Under this perspective, $X_k$ denotes the color of the $k$-th ball drawn with replacement from urn-$x$. Similarly, $Y_k$ is the color of the $k$-th ball drawn with replacement from urn-$y$. Note that based on our formulation, distinct draws are always independent.

The mathematical analysis that follows was motivated by the problem of estimating  the fraction of balls in urn-$x$ with a color that is absent in urn-$y$. We can write this parameter as
\begin{equation}
\label{Eqn:thetaAsympDef}
\theta_{x,y}(\infty):=\sum_{i\in (I_x\setminus I_y)}\Prob_x(i)=\lim_{k\to\infty}\theta_{x,y}(k),
\end{equation}
where
\begin{equation}
\label{Eqn:thetaDef}
\theta_{x,y}(k):=\sum_{i\in I_x}\Prob_x(i)(1-\Prob_y(i))^k=\Prob\left(X_{1} \notin \left\{Y_1,\ldots,Y_k\right\}\right).
\end{equation}
The parameter $\theta_{x,y}(\infty)$ measures the proportion of urn-$x$ which is unique from urn-$y$. On the other hand, $\theta_{x,y}(k)$ is a measure of the effectiveness of $k$-samples from urn-$y$ to determine uniqueness in urn-$x$. This motivates us to refer to the quantity in (\ref{Eqn:thetaAsympDef}) as the \emph{dissimilarity of urn-$x$ from urn-$y$}, and to the quantity in (\ref{Eqn:thetaDef}) as the \emph{average dissimilarity of urn-$x$ relative to $k$-draws from urn-$y$}. Note that these parameters are in general asymmetric in the roles of the urns. In what follows, urns-$x$ and -$y$ are assumed fixed, which motivates us to remove subscripts and write $\theta(k)$ instead of $\theta_{x,y}(k)$.

Unfortunately, one cannot estimate unbiasedly the dissimilarity of one urn from another based on finite samples, as stated in the following result. (See the Materials and Methods section for the proofs of all of our results.)

\begin{thm}
\label{thm:nounbiased}
(No unbiased estimator of dissimilarity.) There is no unbiased estimator of $\theta(\infty)$ based on finite samples from two arbitrary urns-$x$ and -$y$.
\end{thm}

Furthermore, estimating $\theta(\infty)$ accurately without further assumptions on the compositions of urns-$x$ and -$y$ seems a difficult if not impossible task. For instance, arbitrarily small perturbations of urn-$y$ are likely to be unnoticed in a sample of a given size from this urn but may drastically affect the dissimilarity of other urns from urn-$y$. To demonstrate this idea, consider a parameter $0\le\epsilon\le1$ and let $\mathbb{P}_x(1) := 1$, $\mathbb{P}_y(1) := \epsilon$ and $\mathbb{P}_y(2) := (1 - \epsilon)$. If $\epsilon = 0$ then $\theta(\infty) = 1$ while, for each $\epsilon >0$, $\theta(\infty) = 0$.

In contrast with the above, for fixed $k$, $\theta(k)$ depends continuously on $(\Prob_x,\Prob_y)$ e.g. under the metric
\[d\big((\Prob_x,\Prob_y),(\Prob_{x'},\Prob_{y'})\big):=\|\Prob_x-\Prob_{x'}\|+\|\Prob_y-\Prob_{y'}\|,\]
where $\|\nu\|:=\sup_{A\subset\NN}|\nu(A)|=\sum_i|\nu(i)|/2$ denotes the total variation of a signed measure $\nu$ over $\NN$ such that $\nu(\NN)=0$. This is the case because
\begin{eqnarray*}
\left|\sum_i\Prob_x(i)(1-\Prob_y(i))^k-\sum_i\Prob_{x'}(i)(1-\Prob_{y'}(i))^k\right|&\le& \sum_i|\Prob_x(i)-\Prob_{x'}(i)|+k\sum_i|\Prob_y(i)-\Prob_{y'}(i)|,\\
&\le& 2(k+1)\cdot d\big((\Prob_x,\Prob_y),(\Prob_{x'},\Prob_{y'})\big).
\end{eqnarray*}
The above implies that $\theta(k)$ is continuous with respect to any metric equivalent to $d$. Many such metrics can be conceived. For instance, if $(\Prob_x^{m}\times\Prob_y^{n})$ denotes the probability measure associated with $m$ samples with replacement from urn-$x$ that are independent of $n$ samples with replacement from urn-$y$ then $\theta(k)$ is also continuous with respect to any of the metrics $d_{m,n}\big((\Prob_x,\Prob_y),(\Prob_{x'},\Prob_{y'})\big):=\|(\Prob_x^m\times\Prob_y^n)-(\Prob_{x'}^m\times\Prob_{y'}^n)\|$, with $m,n\ge1$, because
\[d\big((\Prob_x,\Prob_y),(\Prob_{x'},\Prob_{y'})\big)/2\le d_{m,n}\big((\Prob_x,\Prob_y),(\Prob_{x'},\Prob_{y'})\big)\le \max\{m,n\}\cdot d\big((\Prob_x,\Prob_y),(\Prob_{x'},\Prob_{y'})\big).\]

Because of the above considerations, we discourage the direct estimation of $\theta(\infty)$ and focus on the problem of estimating $\theta(k)$ accurately.

\section*{Results}

Consider a finite number of draws with replacement $X_1,\ldots,X_{n_x}$ and $Y_1,\ldots,Y_{n_y}$, from urn-$x$ and urn-$y$, respectively, where $n_x,n_y\ge1$ are assumed fixed. Using this data we can estimate $\theta(k)$, for $k=1:n_y$, via the estimator:
\begin{equation}
\label{Eqn:CompEst}
\hat{\theta}(k) := \frac{1}{n_x{n_y \choose k}}\mathop{\sum}\limits_{j=0}^{n_y - k}{n_y - j \choose k} Q(j),
\end{equation}
where
\begin{equation}
\label{def:Q(j)}
Q(j):=\left\{\begin{array}{l}
\hbox{number of indices $i=1:n_x$ such that}\\
\hbox{color $X_i$ occurs $j$-times in $Y_1,\ldots,Y_{n_y}$.}
\end{array}\right.
\end{equation}
We refer to $Q(0),\ldots,Q(n_y)$ as the $Q$-statistics summarizing the data from both urns. Due to the well-known relation: $\sum_{i=1}^ri = r(r+1)/2$, at most $(1+\sqrt{2 n_y})$ of these estimators are non-zero. This sparsity may be exploited in the calculation of the right-hand side of (\ref{Eqn:CompEst}) over a large range of $k$'s.

Our statistic in $\hat\theta(k)$ is the U-statistic associated with the kernel $[\![X_1\notin\{Y_1,\ldots,Y_k\}]\!]$, where $[\![\cdot]\!]$ is used to denote the indicator function of the event within the brackets (Iverson's bracket notation). Following the approach by Halmos in~\cite{Halmos}, we can show that this U-statistic is optimal amongst the unbiased estimators of $\theta(k)$ for $k=1:n_y$. We note that no additional samples from either urn are necessary to estimate $\theta(k)$ unbiasedly over this range when $n_x\ge1$. This contrasts with the estimator in equation~(\ref{ide:Starr}), which requires sample enlargement for unbiased estimation of discovery probability of a sample of size $n$.

\begin{thm}
\label{Thm:UMVUE}
(Minimum variance unbiased estimator.) If $n_x \ge 1$ and $n_y \ge k$ then $\hat{\theta}(k)$ is the unique uniformly minimum variance unbiased estimator of $\theta(k)$. Further, no unbiased estimator of $\theta(k)$ exists for $n_x=0$ or $n_y<k$.
\end{thm}

Our next result shows that $\hat\theta(k)$ converges uniformly in probability to $\theta(k)$ over the largest possible range where unbiased estimation of the later parameter is possible, despite the non-Markovian nature of $\hat\theta(k)$ when applied sequentially over $k$. The result asserts that $\hat\theta(k)$ is likely to be a good approximation of $\theta(k)$, uniformly for $k=1:n_y$, when $n_x$ and $n_y$ are large. The method of proof uses an approach by  Hoeffding~\cite{Hoeffding} for the exact calculation of the variance of a $U$-statistic.

\begin{thm}
\label{Thm:ASConv}
(Uniform convergence in probability.) Independently of how $n_x$ and $n_y$ tend to infinity, it follows for each $\epsilon>0$ that
\begin{equation}
\label{Eqn:ProbConv}
\mathop{\lim}\limits_{n_x,n_y\rightarrow\infty}\mathbb{P}\left(\mathop{\max}\limits_{k=1:n_y}|\hat{\theta}(k)-\theta(k)|>\epsilon\right) = 0.
\end{equation}
\end{thm}

We may estimate the variance of $\hat\theta(k)$ for $k=1:n_y$ via a leave-one-out or also called delete-$1$ jackknife estimator, using an approach studied by Efron and Stein~\cite{EfronStein} and Shao and Wu~\cite{ShaoWu}.

To account for variability in the $x$-data through a leave-one-out jackknife estimate, we require that $n_x\ge 2$ and let
\begin{eqnarray}
\label{Eqn:Vx}
S^2_x(k)& := &\frac{1}{n_x(n_x-1)}\sum\limits_{j=0}^{n_y-k}Q(j)\left(\frac{{n_y-j\choose k}}{{n_y\choose k}} - \hat{\theta}(k)\right)^2.
\end{eqnarray}

On the other hand, to account for variability in the $y$-data, consider for $i\ge1$ and $j\ge0$ the statistics
\begin{equation}
\label{def:M(i,j)}
M(i,j):=\left\{\begin{array}{l}
\hbox{number of colors $c$\, such\, that\, color\, $c$\, occurs\, exactly}\\
\hbox{$i$-times in $(X_1,\ldots,X_{n_x})$ and $j$-times in $(Y_1,\ldots,Y_{n_y})$.}
\end{array}\right.
\end{equation}
Clearly, $\sum_i i\,M(i,j)=Q(j)$; in particular, the $M$-statistics are a refinement of the $Q$-statistics. Define $S^2_y(n_y):=0$ and, for $k<n_y$, define
\begin{eqnarray}
\label{Eqn:Vy}
S^2_y(k)& := &\frac{n_y-1}{n_y}\mathop{\sum}\limits_{i=1}^{n_x}\mathop{\sum}\limits_{j=1}^{n_y-k}j\,M(i,j)\left(i(c_{j-1}(k)-c_j(k)) + \hat{\theta}_y(k) - \hat{\theta}(k)\right)^2,
\end{eqnarray}
where
\begin{eqnarray}
\label{Eqn:CjDef}
c_j(k) &:= & \frac{{n_y-j-1\choose k}}{n_x{n_y-1\choose k}};\\
\label{Eqn:ThetaYDef}
\hat{\theta}_y(k) &:= &\mathop{\sum}\limits_{j=0}^{n_y - k-1}c_j(k)\,Q(j).
\end{eqnarray}

Our estimator of the variance of $\hat\theta(k)$ is obtained by summing the variance attributable to the $x$-data and the $y$-data and is given by
\begin{equation}
\label{Eqn:Vhat}
S^2(k):= S^2_x(k)+ S^2_y(k),
\end{equation}
for $k=1:n_y$; in particular, $S(k)$ is our jackknife estimate of the standard deviation of $\hat{\theta}(k)$.

To assess the quality of $S^2(k)$ as an estimate of the variance of $\hat\theta(k)$ and the asymptotic distribution of the later statistic, we require a few assumptions that rule out degenerate cases. The following conditions are used in the remaining theorems in this section:
\begin{enumerate}
\item[(a)]  $|I_x\cap I_y|<\infty$.
\item[(b)]  there are at least two colors in $(I_x\cap I_y)$ that occur in different proportions in urn-$y$; in particular, the conditional probability $\Prob_y( \cdot\mid I_x\cap I_y)$ is not a uniform distribution.
\item[(c)]  urn-$x$ contains at least one color that is absent in urn-$y$; in particular, $\theta(\infty)>0$.
\item[(d)]  $n_x$ and $n_y$ grow to infinity at a comparable rate i.e. $n_x = \Theta(n_y)$, which means that there exist finite constants $c_1,c_2>0$ such that $c_1 n_y\le n_x\le c_2 n_y$, as $n_x,n_y$ tend to infinity.
\end{enumerate}

Conditions (a-c) imply that $\hat{\theta}(k)$ has a strictly positive variance and that a projection random variable, intermediate between $\hat{\theta}(k)$ and $\theta(k)$, has also a strictly positive variance. The idea of projection is motivated by the analysis of Grams and Serfling in~\cite{GramsSerfling}.

Condition (d) is technical and only used to show that the result in Theorem~\ref{Thm:CLT} holds for the largest possible range of values of $k$ namely, for $k=1:n_y$. See~\cite{HamptonDissertation} for results with uniformity related to Theorem~\ref{Thm:VarAcc}, as well as uniformity results when condition (d) is not assumed.

Because the variance of $\hat\theta(k)$, from now on denoted $\mathbb{V}(\hat\theta(k))$, and its estimate $S^2(k)$ tend to zero as $n_x$ and $n_y$ increase, the unnormalized consistency result is unsatisfactory. As an alternative, we can show that $S^2(k)$ is a consistent estimator relative to $\mathbb{V}(\hat{\theta}(k))$, as stated next.

\begin{thm}
\label{Thm:VarAcc}
(Asymptotic consistency of variance estimation.) If conditions (a)-(c) are satisfied then, for each $k\ge1$ and $\epsilon>0$, it applies that
\begin{align}
\label{Eqn:JackknifeUnif}
\mathop{\lim}\limits_{n_x,n_y\rightarrow\infty}\Prob\left(\left|\frac{S^2(k)}{\mathbb{V}(\hat{\theta}(k))}-1\right|>\epsilon\right)= 0.
\end{align}
\end{thm}

Finally, under conditions (a)-(d), we show that $\hat{\theta}(k)$ is asymptotically normally distributed for all $k=1:n_y$, as $n_x$ and $n_y$ increase at a comparable rate.

\begin{thm}
\label{Thm:CLT}
(Asymptotic normality.) Let $Z\sim\mathcal{N}(0,1)$ i.e. $Z$ has a standard normal distribution. If conditions (a)-(d) are satisfied then
\begin{equation}
\label{Eqn:CLTUniform}
\mathop{\lim}\limits_{n_x,n_y \rightarrow \infty}\,\,\mathop{\max}\limits_{k=1:n_y}\left|\Prob\left(\frac{\hat{\theta}(k)-\theta(k)}{\sqrt{\mathbb{V}(\hat\theta(k)}}\le t\right)-\Prob(Z\le t)\right|=0,
\end{equation}
for all real number $t$.
\end{thm}

The non-trivial aspect of the above result is the asymptotic normality of $\hat\theta(k)$ when $k=\Theta(n_y)$, e.g. $\hat\theta(n_y)$, as the results we have found in the literature~\cite{Hoeffding, Ahmad, CallaertJanssen} only guarantee the asymptotic normality of our estimator of $\theta(k)$ for fixed $k$. We note that, due to Slutsky's theorem~\cite{Slutsky}, it follows from (\ref{Eqn:JackknifeUnif}) and (\ref{Eqn:CLTUniform}) that the ratio
\[\frac{\hat\theta(k)-\theta(k)}{S(k)}\]
has, for fixed $k$, approximately a standard normal distribution when $n_x$ and $n_y$ are large and of a comparable order of magnitude.

\section*{Discussion}
As proof of concept, we use our estimators to analyze data from the Human Microbiome Project (HMP)~\cite{HMPAnalysisPaper}. In particular, our samples are V35 16S rRNA data, processed by Qiime into an operational taxonomic unit (OTU) count table format (see File S1 in Supporting Information). Each of the $266$ samples analyzed have more than $5000$ successfully identified bacteria (see File S2 in Supporting Information). We sort these samples by the body location metadata describing the origin of the sample. This sorting yields the assignments displayed in Table~\ref{tab:1}.

We present our estimates of $\hat{\theta}(n_y)$ for all $266\cdot265$ possible sample comparisons in Figure~\ref{Fig:ThetaMap}, i.e., we estimate the average dissimilarity of sample-$x$ relative to the full sample-$y$. Due to (\ref{Eqn:CompEst}), observe that $\hat\theta(n_y)=Q(0)/n_y$. At the given sample sizes, we can differentiate four broad groups of environments: stool, vagina, oral/throat and skin/nostril. We differentiate a larger proportion of oral/throat bacteria found in stool than stool bacteria found in the oral/throat environments. We may also differentiate the throat, gingival and saliva samples, but cannot reliably differentiate between tongue and throat samples or between the subgingival and supragingival plaques. On the other hand, the stool samples have larger proportions of unique bacteria relative to other stool samples of the same type, and  vaginal samples also have this property. In contrast the skin/nostril samples have relatively few bacteria that are not identified in other skin/nostril samples.

The above effects may be a property of the environments from which samples are taken, or an effect of noise from inaccurate estimates due to sampling. To rule out the later interpretation, we show estimates of the standard deviation of $\hat{\theta}(n_y)$ based on the jackknife estimator $S^2(n_y)$ from (\ref{Eqn:Vhat}) in Figure~\ref{Fig:ThetaVarMap}. As $S_{n_y}^2(n_y)$ is zero, the error estimate is given by $S_x(n_y)$. We see from (\ref{Eqn:Vx}), with $k=n_y$, that 
\[S(n_y)=\sqrt{\frac{\hat\theta(n_y)\cdot(1-\hat\theta(n_y))}{n_x-1}}.\]
Assuming a normal distribution and an accurate jackknife estimate of variance, $\theta(n_y)$ will be in the interval $\hat{\theta}(n_y)\pm 0.01$ with at least approximately 95\% confidence, for any choice of sample comparisons in our data; in particular, on a linear scale, we expect at least 95\% of the estimates in Figure~\ref{Fig:ThetaMap} to be accurate in at least the first two digits.

As we mentioned earlier, estimating $\theta(\infty)$ accurately is a difficult problem. We end this section with two heuristics to assess how representative $\hat\theta(n_y)$ is of $\theta(\infty)$, when urn-$y$ has at least two colors and at least one color in common with urn-$x$. First, observe that:
\begin{align}
\label{Eqn:ThetaKApproxStar}
\theta(k)&=\theta(\infty)+\sum_{i\in(I_x\cap I_y)}\Prob_x(i)(1-\Prob_y(i))^k.
\end{align}
In particular, $\theta(k)$ is a strictly concave-up and monotonically decreasing function of the real-variable $k\ge 0$. Hence, if $\theta(n_y)$ is close to the asymptotic value $\theta(\infty)$, then $\theta(n_y)-\theta(n_y-1)$ should be of small magnitude. We call the later quantity the \emph{discrete derivative} of $\theta(k)$ at $k=n_y$. Since we may estimate the discrete derivative from our data, the following heuristic arises: \emph{relatively large values of $|\hat\theta(n_y)-\hat\theta(n_y-1)|$ are evidence that $\hat\theta(n_y)$ is not a good approximation of $\theta(\infty)$.}

Figure~\ref{Fig:DThetaMap} shows the heat map of  $|\hat\theta(n_y)-\hat\theta(n_y-1)|$ for each pair of samples. These estimates are of order $10^{-5}$ for the majority of the comparisons, and spike to $10^{-4}$ for several sample-$y$ of varied environment types, when sample-$x$ is associated with a skin or vaginal sample. In particular, further sampling effort from environments associated with certain vaginal, oral or stool samples are likely to reveal bacteria associated with broadly defined skin or vaginal environments.

Another heuristic may be more useful to assess how close $\hat\theta(n_y)$ is to $\theta(\infty)$, particularly when the previous heuristic is inconclusive. As motivation, observe that $\theta(k)=\theta(\infty)+\Theta(\rho^k)$, because of the identity in (\ref{Eqn:ThetaKApproxStar}), where
\[\rho:=1-\min_{i\in(I_x\cap I_y)}\Prob_y(i).\]
Furthermore, $\log(\theta(k-1)-\theta(k))=k(\ln\rho)+c+o(1)$, where $c$ is certain finite constant. We can justify this approximation only when $\log(\theta(k-1)-\theta(k))$ is well approximated by a linear function of $k$, in which case we let $\hat\rho$ denote the estimated value for $\rho$ obtained from the linear regression. Since $0\le\theta(n_y)-\theta(\infty)\le\rho^{n_y}$, the following more precise heuristic comes to light: \emph{$\hat\theta(n_y)$ is a good approximation of $\theta(\infty)$ if the linear regression of $\log|\hat\theta(k-1)-\hat\theta(k)|$ for $k$ near $n_y$ gives a good fit, $S(n_y)$ is small relative to $\hat\theta(n_y)$, and $\hat\rho^{n_y}$ is also small.}

To fix ideas we have applied the above heuristic to three pairs of samples: $(255,176)$, $(200,139)$ and $(100,10)$, with each ordered pair denoting urn-$x$ and urn-$y$, respectively. As seen in Table~\ref{tab:2} for these three cases, $\hat{\theta}(n_y)$ is at least 14-times larger than $S(n_y)$; in particular, due to the asymptotic normality of the later statistic, an appropriate use of the heuristic is reduced to a good linear fit and a small $\hat\rho^{n_y}$ value. In all three cases, $\hat\rho$ was computed from the estimates $\hat\theta(k)$, with $k=5001:n_y$.

For the $(255,176)$-pair, $\hat\rho^{n_y}$ and the regression error, measured as the largest absolute residual associated with the best linear fit, are zero to machine precision, suggesting that $\hat{\theta}(n_y)=0.9998$ is a good approximation of $\theta(\infty)$. This is reinforced by the blue plot in Figure~\ref{Fig:ThetaCurveMap}. On the other hand, for the $(200,139)$-pair, the regression error is small, suggesting that the linear approximation $\log(\hat\theta(k-1)-\hat\theta(k))$ is good for $k=5001:n_y$. However, because $\hat\rho^{n_y}=0.9997$, we cannot guarantee that $\hat\theta(n_y)$ is a good approximation of $\theta(\infty)$. In fact, as seen in the red-plot in Figure~\ref{Fig:ThetaCurveMap}, $\hat\theta(k)$, with $k=1:n_y$, exposes a steady and almost linear decay that suggests that $\theta(\infty)$ may be much smaller than $\hat\theta(n_y)$. Finally, for the $(100,10)$-pair, the regression error is large and the heuristic is therefore inconclusive. Due to the green-plot in Figure~\ref{Fig:ThetaCurveMap}, the lack of fit indicates that the exponential rate of decay of $\theta(k)$ to $\theta(\infty)$ has not yet been captured by the data from these urns. Note that the heuristic based on the discrete derivative shows no evidence that $\hat\theta(n_y)$ is far from $\theta(\infty)$.

\section*{Materials and Methods}
Here we prove the theorems given in the Results section. The key idea to prove each theorem may be summarized as follows.

To show Theorem~\ref{thm:nounbiased}, we identify pairs of urns for which unbiased estimation of $\theta(\infty)$ is impossible for any statistic. To show Theorem~\ref{Thm:UMVUE}, we exploit the diversity of possible urn distributions to show that there are relatively few unbiased estimators of $\theta(k)$ and, in fact, there is a single unbiased estimator $\hat\theta(k)$ that is symmetric on the data. The uniqueness of the symmetric estimator is obtained via a completeness argument: a symmetric statistic having expected value zero is shown to correspond to a polynomial with identically zero coefficients, which themselves correspond to values returned by the statistic when presented with specific data. The symmetric estimator is a U-statistic in that it corresponds to an average of unbiased estimates of $\theta(k)$, based on all possible sub-samples of size $1$ and $k$ from the samples of urn-$x$ and -$y$, respectively. As any asymmetric estimator has higher variance than a corresponding symmetric estimator, the symmetric estimator must be the UMVUE. 

To show Theorem~\ref{Thm:ASConv} we use bounds on the variance of the U-statistic and show that, uniformly for relatively small $k$, $\hat\theta(k)$ converges to $\theta(k)$ in the $\mathcal{L}^2$-norm. In contrast, for relatively large values of $k$, we exploit the monotonicity of $\theta(k)$ and $\hat{\theta}(k)$ to show uniform convergence.

Finally, theorems~\ref{Thm:VarAcc} and~\ref{Thm:CLT} are shown using an approximation of $\hat{\theta}(k)$ by sums i.i.d. random variables, as well as results concerning the variance of both $\hat{\theta}(k)$ and its approximation. In particular, the approximation satisfies the hypotheses the Central Limit Theorem and Law of Large Numbers, which we use to transfer these results to $\hat{\theta}(k)$.

In what follows, $\D$ denotes the set of all probability distributions that are finitely supported over $\NN$.\\

\noindent\textbf{Proof of Theorem~\ref{thm:nounbiased}.} Consider in $\D$ probability distributions of the form $\Prob_x(1)=1$, $\Prob_y(1)=u$ and $\Prob_y(2)=(1-u)$, where $0\le u\le1$ is a given parameter. Any statistic $h(\cdot)$ which takes as input $n_x$ draws from urn-$x$ and $n_y$ draws from urn-$y$ has that $\Esp(h(X_1,\ldots,X_{n_x},Y_1,\ldots,Y_{n_y}))$ is a polynomial of degree at most $n_y$ in the variable $u$; in particular, it is a continuous function of $u$ over the interval $[0,1]$. Since $\theta(\infty)=[\![u=0]\!]$ has a discontinuity at $u=0$ over this interval, there exists no estimator of $\theta(\infty)$ that is unbiased over pairs of distributions in $\D$.\hfill$\Box$\\

We use lemmas~\ref{Lem:ThetaMin}-\ref{Lem:f=0} to first show Theorem~\ref{Thm:UMVUE}. The method of proof of this theorem follows an approach similar to the one used by Halmos~\cite{Halmos} for single distributions, which we extend here naturally to the setting of two distributions.

Our next result implies that no uniformly unbiased estimator of $\theta(k)$ is possible when using less than one sample from urn-$x$ and $k$ samples from urn-$y$.

\begin{lem}
\label{Lem:ThetaMin}
If $g(X_{1},\ldots,X_{m},Y_{1},\ldots,Y_{n})$ is unbiased for $\theta(k)$ for all $\Prob_x,\Prob_y\in\D$, then $m\ge1$ and $n\ge k$.
\end{lem}
\begin{proof}
Consider in $\D$ probability distributions of the form $\Prob_x(1) = u$, $\Prob_x(2) = (1-u)$, $\Prob_y(1) = v$ and $\Prob_y(2) = (1-v)$, where $0\le u,v\le 1$ are arbitrary real numbers. Clearly, $\Esp[g(X_1,\ldots,X_m,Y_1,\ldots,Y_n)]$ is a linear combination of polynomials of degree $m$ in $u$ and $n$ in $v$ and, as a result, it is a polynomial of degree at most $m$ in $u$ and $n$ in $v$. Since $\theta(k)=u(1-v)^k + (1-u) v^k$ has degree $1$ in $u$ and $k$ in $v$, and $g(X_1,\ldots,X_m,Y_1,\ldots,Y_n)$ is unbiased for $\theta(k)$, we conclude that $1\le m$ and $k\le n$.
\end{proof}

The form of $\hat{\theta}(k)$ given in equation~(\ref{Eqn:CompEst}) is convenient for computation but, for mathematical analysis, we prefer its $U$-statistic form associated with the kernel function $(x,y_1,\ldots,y_k)\to[\![x\notin\{y_1,\ldots,y_k\}]\!]$.

In what follows, $S_{k,n}$ denotes the set of all functions $\sigma:\{1,\ldots,k\}\to\{1,\ldots,n_y\}$ that are one-to-one.

\begin{lem}
\label{Lem:UStatDef}
\begin{equation}
\label{Eqn:UStatDef}
\hat{\theta}(k) = \frac{1}{n_x|S_{k,n_y}|}\sum\limits_{i=1}^{n_x}\sum\limits_{\sigma\in S_{k,n_y}}[\![X_i\notin\{Y_{\sigma(1)},\ldots,Y_{\sigma(k)}\}]\!],
\end{equation}
where $|S_{k,n_y}|=k!{n_y \choose k}$.
\end{lem}
\begin{proof}
Fix $1\le i\le n_x$ and suppose that color $X_i$ occurs $j$-times in $Y_1,\ldots,Y_{n_y}$. If $j>(n_y-k)$ then any sublist of size $k$ of $Y_1,\ldots,Y_{n_y}$ contains $X_i$, hence $[\![X_i\notin\{Y_{\sigma(1)},\ldots,Y_{\sigma(k)}\}]\!]=0$, for all $\sigma\in S_{k,n_y}$. On the other hand, if $j\le(n_y-k)$ then $\sum_{\sigma\in S_{k,n_y}}[\![X_i\notin\{Y_{\sigma(1)},\ldots,Y_{\sigma(k)}]\!]=k!{n_y-j\choose k}$. Since the rightmost sum only depends on the number of times that color $X_i$ was observed in $Y_1,\ldots,Y_{n_y}$, we may use the $Q$-statistics defined in equation (\ref{def:Q(j)}) to rewrite:
\[\frac{1}{n_x|S_{k,n_y}|}\sum\limits_{i=1}^{n_x}\sum\limits_{\sigma\in S_{k,n_y}}[\![X_i\notin\{Y_{\sigma(1)},\ldots,Y_{\sigma(k)}\}]\!]=\frac{1}{n_x{n_y\choose k}}\sum_{j=0}^{n_y-k}{n_y-j\choose k} Q(j).\]
The right-hand side above now corresponds to the definition of $\hat{\theta}(k)$ given in equation~(\ref{Eqn:CompEst}).
\end{proof}

In what follows, we say that a function $f:\NN^{n_x+n_y}\to\RR$ is \emph{$(n_x,n_y)$-symmetric} when
\[f(x_1,\ldots,x_{n_x};y_1,\ldots,y_{n_y})=f(x_{\sigma(1)},\ldots,x_{\sigma(n_x)};y_{\sigma'(1)},\ldots,y_{\sigma'(n_y)}),\]
for all $x_1,\ldots,x_{n_x},y_1,\ldots,y_{n_y}\in\NN$ and permutations $\sigma$ and $\sigma'$ of $1,\ldots,n_x$ and $1,\ldots,n_y$, respectively. Alternatively, $f$ is $(n_x,n_y)$-symmetric if and only if it may be regarded a function of $(x_{(1\ldots n_x)},y_{(1\ldots n_y)})$, where $x_{(1\ldots n_x)}$ and $y_{(1\ldots n_y)}$ correspond to the order statistics $x_{(1)},\ldots,x_{(n_x)}$ and $y_{(1)},\ldots,y_{(n_y)}$, respectively. Accordingly, a statistic of $(X_1,\ldots,X_{n_x},Y_1,\ldots,Y_{n_y})$ is called \emph{$(n_x,n_y)$-symmetric} when it may be represented in the form $f(X_1,\ldots,X_{n_x},Y_1,\ldots,Y_{n_y})$, for some $(n_x,n_y)$-symmetric function $f$.
It is immediate from Lemma~\ref{Lem:UStatDef} that $\hat\theta(k)$ is $(n_x,n_y)$-symmetric. 

The next result asserts that the variance of any non-symmetric unbiased estimator of $\theta(k)$ may be reduced by a corresponding symmetric unbiased estimator. The proof is based on the well-known fact that conditioning preserves the mean of a statistic and cannot increase its variance.

\begin{lem}
\label{Lem:SymUMVUE}
An asymmetric unbiased estimator of $\theta(k)$ that is square-integrable has a strictly larger variance than a corresponding $(n_x,n_y)$-symmetric unbiased estimator.
\end{lem}
\begin{proof}
Let $\F$ denote the sigma-field generated by the random vector $(X_{(1\ldots n_x)};Y_{(1\ldots n_y)})$ and suppose that the statistic $T=f(X_1,\ldots,X_{n_x},Y_1,\ldots,Y_{n_y})$ is unbiased for $\theta(k)$ and square-integrable. In particular, $U=\Esp[T\!\mid\!\F]$ is a well-defined statistic and there is an $(n_x,n_y)$-symmetric function $g:\NN^{n_x+n_y}\to\RR$ such that $U=g(X_1,\ldots,X_{n_x};Y_1,\ldots,Y_{n_y})$. Clearly, $U$ is unbiased for $\theta(k)$ and $(n_x,n_y)$-symmetric. Since $\Esp(T^2)<+\infty$, Jensen's inequality for conditional expectations~\cite{Durrett} implies that $\Esp(U^2)\le\Esp(T^2)$, with equality if and only if $T$ is $(n_x,n_y)$-symmetric.
\end{proof}

Since $\hat\theta(k)$ is $(n_x,n_y)$-symmetric and bounded, the above lemma implies that if an UMVUE for $\theta(k)$ exists then it must be $(n_x,n_y)$-symmetric. Next, we show that there is a unique symmetric and unbiased estimator of $\theta(k)$, which immediately implies that $\hat\theta(k)$ is the UMVUE.

In what follows, $k_1,k_2\ge0$ denote integers. We say that a polynomial $Q(u_1,\ldots,u_m;v_1,\ldots,v_n)$ is \emph{$(k_1,k_2)$-homogeneous} when it is a linear combination of polynomials of the form $\prod_{i=1}^mu_i^{m_i}\prod_{j=1}^nv_j^{n_j}$, with $\sum_{i=1}^mm_i=k_1$ and $\sum_{j=1}^nn_j=k_2$. Furthermore, we say that $Q$ satisfies the \emph{partial vanishing condition} if $Q(u_1,\ldots,u_m;v_1,\ldots,v_n)=0$ whenever $u_1,\ldots,u_m,v_1,\ldots,v_n\ge0$, $\sum_{i=1}^mu_i=1$ and $\sum_{i=1}^nv_i=1$.

The next lemma is an intermediate step to show that a $(k_1,k_2)$-homogeneous polynomial which satisfies the partial vanishing condition is the zero polynomial, which is shown in Lemma~\ref{Lem:Q2=0}.

\begin{lem}
\label{Lem:Q=0 pos}
If $Q$ is a $(k_1,k_2)$-homogeneous polynomial in the real variables $u_1,\ldots,u_m,v_1,\ldots,v_n$, with $m,n\ge1$, that satisfies the partial vanishing condition, then $Q(u_1,\ldots,u_m;v_1,\ldots,v_n)=0$ whenever $u_1,\ldots,u_m,v_1,\ldots,v_n\ge0$, $\sum_{i=1}^mu_i>0$ and $\sum_{i=1}^nv_i>0$.
\end{lem}
\begin{proof}
Fix $u_1,\ldots,u_m,v_1,\ldots,v_n\ge0$ such that $\sum_{i=1}^mu_i>0$ and $\sum_{i=1}^nv_i>0$ and observe that
\[Q(u_1,\ldots,u_m;v_1,\ldots,v_n):=\left(\sum_{i=1}^mu_i\right)^{k_1}\left(\sum_{i=1}^nv_i\right)^{k_2} Q\left(\frac{u_1}{\sum_{i=1}^mu_i},\ldots,\frac{u_m}{\sum_{i=1}^mu_i};\frac{v_1}{\sum_{i=1}^nv_i},\ldots,\frac{v_n}{\sum_{i=1}^nv_i}\right),\]
because $Q$ is a $(k_1,k_2)$-homogeneous polynomial. Notice now that the right hand-side above is zero because $Q$ satisfies the partial vanishing condition.
\end{proof}

\begin{lem}
\label{Lem:Q2=0}
Let $Q$ be a $(k_1,k_2)$-homogeneous polynomial in the real variables $u_1,\ldots,u_m,v_1,\ldots,v_n$, with $m,n\ge1$. If $Q$ satisfies the partial vanishing condition then $Q=0$ identically.
\end{lem}
\begin{proof}
We prove the lemma using structural induction on $(m,n)$ for all $k_1,k_2\ge0$.

If $m=n=1$ then a $(k_1,k_2)$-homogeneous polynomial $Q(u_1,v_1)$ must be of the form $cu^{k_1}_1v^{k_2}_1$, for an appropriate constant $c$. As such a polynomial satisfies the partial-vanishing condition only when $c=0$, the base case for induction is established.

Next, consider a $(k_1,k_2)$-homogeneous polynomial $Q(u_1,\ldots,u_m;v_1,\ldots,v_n,v_{n+1})$, with $m,n\ge1$, that satisfies the partial vanishing condition, and let $d$ denote its degree with respect to the variable $v_{n+1}$. In particular, there are polynomials $Q_0,\ldots,Q_d$ in the variables $u_1,\ldots,u_m,v_1,\ldots,v_n$ such that
\[Q(u_1,\ldots,u_m;v_1,\ldots,v_n,v_{n+1})=\sum_{i=0}^d Q_i(u_1,\ldots,u_m;v_1,\ldots,v_n) v_{n+1}^i.\]
Now fix $u_1,\ldots,u_m,v_1,\ldots,v_n\ge0$ such that $\sum_{i=1}^mu_i>0$ and $\sum_{i=1}^nv_i>0$. Because $Q$ satisfies the partial vanishing condition, Lemma~\ref{Lem:Q=0 pos} implies that $\sum_{i=0}^d Q_i(u_1,\ldots,u_m;v_1,\ldots,v_n) v_{n+1}^i=0$ for all $v_{n+1}>0$. In particular, for each $i$, $Q_i(u_1,\ldots,u_m;v_1,\ldots,v_n)=0$ whenever $u_1,\ldots,u_m,v_1,\ldots,v_n\ge0$, $\sum_{i=1}^mu_i>0$ and $\sum_{i=1}^nv_i>0$. Thus each $Q_i$ satisfies the partial vanishing condition. Since $Q_i$ is a $(k_1,k_2-i)$-homogeneous polynomial, the inductive hypothesis implies that $Q_i=0$ identically and hence $Q=0$ identically. The same argument shows that if $Q(u_1,\ldots,u_m,u_{m+1};v_1,\ldots,v_n)$, with $m,n\ge1$, is a $(k_1,k_2)$-homogeneous polynomial that satisfies the partial vanishing condition then $Q=0$ identically, completing the inductive proof of the lemma.
\end{proof}

Our final result{before proving Theorem~\ref{Thm:UMVUE} implies that $\theta(k)$ cannot admit more than one symmetric and unbiased estimator. Its proof depends on the variety of distributions in $\D$, and uses the requirement that our estimator must be unbiased for any pair of distributions chosen from $\D$.

\begin{lem}
\label{Lem:f=0}
If $f$ is an $(n_x,n_y)$-symmetric function such that $\Esp[f(X_1,\ldots,X_{n_x},Y_1,\ldots,Y_{n_y})]=0$, for all $\Prob_x,\Prob_y \in \D$, then $f=0$ identically.
\end{lem}
\begin{proof}
Consider a point $\vec{z}=(x_1,\ldots,x_{n_x},y_1,\ldots,y_{n_y})\in\mathbb{N}^{n_x+n_y}$ and define $m_1$ and $m_2$ as the cardinalities of the sets $\{x_1,\ldots,x_{n_x}\}$ and $\{y_1,\ldots,y_{n_y}\}$, respectively. Furthermore, let $x^{\prime}_1,\ldots,x^{\prime}_{m_1}$ denote the distinct elements in the set $\{x_1,\ldots,x_{n_x}\}$ and define $m_{1,i}$ to be the number of times that $x^{\prime}_i$ appears in this set. Furthermore, let $\Prob_x\in\D$ be a probability distribution such that $\Prob_x(\{x_1',\ldots,x_{m_1}'\})=1$ and define $p_{1,i}:=\Prob_x(x_i')$. In a completely analogous manner define $y^{\prime}_1,\ldots,y^{\prime}_{m_2}$, $m_{2,j}$, $\Prob_y$ and $p_{2,j}$.

Notice that $\Esp[f(\vec{Z}_m)]$ is a polynomial in the real variables $p_{1,1},\ldots,p_{1,m_1},p_{2,1},\ldots,p_{2,m_2}$ that satisfies the hypothesis of Lemma~\ref{Lem:Q2=0}; in particular, this polynomial is identically zero. However, because $f$ is $(n_x,n_y)$-symmetric, the coefficient of $\mathop{\prod}_{i=1}^{2}\mathop{\prod}_{j=1}^{m_i}p^{m_{i,j}}_{i,j}$ in $\Esp[f(\vec{Z}_m)]$ is
\[f(\vec{z})\,{n_x\choose m_{1,1};\ldots;m_{1,m_1}}\,{n_y\choose m_{2,1};\ldots;m_{2,m_2}},\]
implying that $f(\vec{z})=0$.
\end{proof}

\noindent\textbf{Proof of Theorem~\ref{Thm:UMVUE}.} From Lemma~\ref{Lem:SymUMVUE}, as we mentioned already, if the UMVUE for $\theta(k)$ exists then it must be $(n_x,n_y)$-symmetric. Suppose there are two $(n_x,n_y)$-symmetric functions such that $f(X_1,\ldots,X_{n_x};Y_1,\ldots,Y_{n_y})$ and $g(X_1,\ldots,X_{n_x};Y_1,\ldots,Y_{n_y})$ are unbiased for $\theta(k)$. Applying Lemma~\ref{Lem:f=0} to $(f-g)$ shows that $f=g$, and $\theta(k)$ admits therefore a unique symmetric and unbiased estimator. From Lemma~\ref{Lem:UStatDef}, $\hat{\theta}(k)$ is $(n_x,n_y)$-symmetric and unbiased for $\theta(k)$ hence it is the UMVUE for $\theta(k)$. From Lemma~\ref{Lem:ThetaMin}, it follows that no unbiased estimator of $\theta(k)$ exists for $n_x=0$ or $n_y<k$. \hfill$\Box$\\

Our next goal is to show Theorem~\ref{Thm:ASConv}, for which we prove first lemmas~\ref{Lem:ChooseRelation}-\ref{Lem:L2Unif}. We note that the later lemma applies in a much more general context than our treatment of dissimilarity.

\begin{lem}
\label{Lem:ChooseRelation}
If, for each $n\ge1$, $k_n\ge1$ is an integer such that $k_n^2=o(n)$ then
\begin{align}
\label{Eqn:Choose2}
\frac{{k\choose 1}{n-k\choose k-1}}{{n\choose k}}&=\frac{k^2}{n}+O\left(\frac{k^4}{n^2}\right);\\
\label{Eqn:Choose1}
\sum\limits_{j=2}^{k}\frac{{k\choose j}{n-k\choose k-j}}{{n\choose k}}&=O\left(\frac{k^4}{n^2}\right);
\end{align}
uniformly for $k=1:k_n$ as $n\to\infty$.
\end{lem}
\begin{proof}
First observe that for all $n$ sufficiently large and $k=1:k_n$, it applies that
\begin{align*}
\frac{{k\choose 1}{n-k\choose k-1}}{{n\choose k}}=\frac{k^2}{n}\mathop{\prod}\limits_{i=0}^{k-2}\left(1-\frac{k-1}{n-1-i}\right)=\frac{k^2}{n}\exp\left\{\sum_{i=0}^{k-2}\log\left(1-\frac{k-1}{n-1-i}\right)\right\}.
\end{align*}
Note that $-x/(1-x)\le\log(1-x)\le-x$, for all $0\le x<1$. As a result, we may bound the exponential factor on the right-hand side above as follows:
\[e^{-(k-1)^2/(n-2k+2)}\le\exp\left\{\sum_{i=0}^{k-2}\log\left(1-\frac{k-1}{n-1-i}\right)\right\}\le e^{-(k-1)^2/(n-1)}.\]
Since $e^{-(k-1)^2/(n-2k+2)}=1+O(k^2/n)$ and $e^{-(k-1)^2/(n-1)}=1+O(k^2/n)$, uniformly for all $k=1:k_n$ as $n\to\infty$, (\ref{Eqn:Choose2}) follows.

To show (\ref{Eqn:Choose1}), first note the combinatorial identity
\begin{equation}
\label{ide:CombIdent}
\mathop{\sum}\limits_{j=0}^{k}\frac{{k\choose j}{n-k\choose k-j}}{{n\choose k}}=1.
\end{equation}
Proceeding in an analogous manner as we did to show (\ref{Eqn:Choose2}), we see now that the term associated with the index $j=0$ in the above summation satisfies that
\[e^{-k^2/(n-2k+1)}\le\frac{{k\choose 0}{n-k\choose k}}{{n\choose k}}\le e^{-k^2/n},\]
for all $n$ sufficiently large and $k=1:k_n$. Since $e^{-k^2/(n-2k+1)}=1-k^2/n+O(k^4/n^2)$ and $e^{-k^2/n}=1-k^2/n+O(k^4/n^2)$, the above inequalities together with (\ref{Eqn:Choose2}) and (\ref{ide:CombIdent}) establish (\ref{Eqn:Choose1}).
\end{proof}

\begin{lem}
\label{Lem:L2Unif}
Define $\lambda(k):=\mathbb{E}(h(X_1,Y_1,\ldots,Y_k))$, where $h(x_1,y_1,\ldots,y_k)$ is a bounded $(1,k)$-symmetric function, and let
\[\hat\lambda(k)=\hat{\lambda}_{n_x,n_y}(k):=\frac{1}{n_x|S_{k,n_y}|}\sum_{i=1}^{n_x}\sum_{\sigma\in S_{k,n_y}}h(X_i,Y_{\sigma(1)},\ldots,Y_{\sigma(k)})\]
be the U-statistic of $\lambda(k)$ associated with $n_x$ draws from urn-$x$ and $n_y$ draws from urn-$y$; in particular, $\mathbb{E}(\hat\lambda(k))=\lambda(k)$. Furthermore, assume that
\begin{enumerate}[I]
\item[(i)] $0\le h\le 1$,
\item[(ii)] there is a function $f:I_x\to[0,1]$ such that $\mathop{\lim}\limits_{k\rightarrow\infty}h(X_1,Y_1,\ldots,Y_k)\mathop{=}\limits^{\mbox{a.s}}f(X_1)$,
\item[(iii)] $\hat{\lambda}(k)\ge\hat{\lambda}(k+1)$; in particular, $\lambda(k)\ge\lambda(k+1)$.
\end{enumerate}
Under the above assumptions, it follows that
\begin{align*}
\mathop{\lim}\limits_{n_x,n_y\rightarrow\infty}\mathbb{E}\left(\mathop{\max}\limits_{k=1:n_y}\left|\hat{\lambda}(k)-\lambda(k)\right|^2\right)&= 0.
\end{align*}
\end{lem}

\begin{proof}
Define $k_n:=1+\min\big\{\lfloor n_x^{1/2}\rfloor,\lfloor \log(n_y)\rfloor\big\}$; in particular, $1\le k_n\le n_y$ and $k_n\to\infty$, $k_n=o(n_x)$ and $k_n^p=o(n_y)$, for any $p>0$, as $n_x,n_y\to\infty$. The proof of the theorem is reduced to show that
\begin{align}
\label{Eqn:ConsL2Large}
\mathop{\lim}\limits_{n_x,n_y\rightarrow\infty}\mathbb{E}\left(\mathop{\max}\limits_{k=k_n:n_y}|\hat{\lambda}(k)-\lambda(k)|^2\right)&=0;\\
\label{Eqn:ConsL2Small}
\mathop{\lim}\limits_{n_x,n_y\rightarrow\infty}\mathbb{E}\left(\mathop{\max}\limits_{k=1:k_n}|\hat{\lambda}(k)-\lambda(k)|^2\right)&=0.
\end{align}

Next, we compute the variance of $\hat\lambda(k)$ following an approach similar to Hoeffding~\cite{Hoeffding}. Because $h$ is $(1,k)$-symmetric, a tedious yet standard calculation shows that
\begin{align}
\label{Eqn:LambdaVar}
\mathbb{V}\big(\hat{\lambda}(k)\big)&=\frac{\mathop{\sum}\limits_{j=0}^{k}{k\choose j}{n_y-k\choose k-j}((n_x-1)\xi_{0,j}(k) + \xi_{1,j}(k))}{n_x{n_y\choose k}},
\end{align}
where
\begin{eqnarray}
\label{def:xi0jkGeneral} \xi_{0,j}(k)&:=&\mathbb{V}\big(\mathbb{E}(h(X_1,Y_1,\ldots,Y_k)|Y_1,\ldots,Y_j)\big);\\
\label{def:xi1jkGeneral} \xi_{1,j}(k)&:=&\mathbb{V}\big(\mathbb{E}(h(X_1,Y_1,\ldots,Y_k)|X_1,Y_1,\ldots,Y_j)\big).
\end{eqnarray}

Clearly, $\xi_{1,j}(k)\le 1$. On the other hand, if $W$ is any random variable with finite expectation and $\mathcal{F}_1\subset\mathcal{F}_2$ are sigma-fields then $\mathbb{V}(\Esp(W|\mathcal{F}_1))\le\mathbb{V}(\Esp(W|\mathcal{F}_2))$, due to well-known properties of conditional expectations~\cite{Durrett}. In particular, for each $0\le j\le k$, we have that
\begin{equation}
\label{ide:sosouseful}
\xi_{0,j}(k)\le\xi_{0,k}(k),\hbox{ and }\xi_{1,j}(k)\le\xi_{1,k}(k).
\end{equation}
Consequently,~(\ref{Eqn:LambdaVar}) implies that
\begin{equation}
\label{ineq:Var20}
\mathbb{V}\big(\hat{\lambda}(k)\big)\le\xi_{0,k}(k)+\frac{1}{n_x}.
\end{equation}

We claim that
\begin{equation}
\label{ineq:Var21}
\lim_{k\to\infty}\xi_{0,k}(k)=0.
\end{equation}
Indeed, using an argument similar as above, we find that
\begin{eqnarray*}
\xi_{0,k}(k)&=&\mathbb{V}\Big(\Esp\big(h(X_1,Y_1,\ldots,Y_k)-f(X_1)\mid Y_1,\ldots,Y_k\big)\Big),\\
&\le&\mathbb{V}\Big(\Esp\big(h(X_1,Y_1,\ldots,Y_k)-f(X_1)\mid Y_1,\ldots,Y_k,X_1\big)\Big),\\
&=&\mathbb{V}\big(h(X_1,Y_1,\ldots,Y_k)-f(X_1)\big).
\end{eqnarray*}
Due to assumptions (\textit{i})-(\textit{ii}) and the Bounded Convergence Theorem, the right-hand side above tends to $0$, and the claim follows.

It follows from (\ref{ineq:Var20}) and (\ref{ineq:Var21}) that
\[\lim_{k\to\infty}\mathbb{V}\big(\hat\lambda(k)\big)=0.\]
Finally, because of assumption (\textit{iii}),
\begin{align*}
\mathbb{E}\left(\mathop{\max}\limits_{k=k_n:n_y}|\hat{\lambda}(k)-\lambda(k)|^2\right)\le 2|{\lambda}(k_n)-\lambda(n_y)|^2+2\mathbb{V}(\hat\lambda(k_n))+2\mathbb{V}(\hat\lambda(n_y)).
\end{align*}
Since each term on the right-hand side above tends to zero as $n_x,n_y\rightarrow\infty$, (\ref{Eqn:ConsL2Large}) follows.

We now show (\ref{Eqn:ConsL2Small}). As $\xi_{i,j}(k)\le1$ and $\xi_{0,0}(k)=0$, it follows by (\ref{Eqn:LambdaVar}) and Lemma~\ref{Lem:ChooseRelation} that
\begin{align*}
\mathbb{V}\left(\hat{\lambda}(k)\right)&\le1-\frac{{n_y-k\choose k}}{{n_y\choose k}}+\frac{1}{n_x}=\frac{1}{n_x}+\sum_{j=1}^k\frac{{k\choose j}{n_y-k\choose k-j}}{{n_y\choose k}}=\frac{1}{n_x}+\frac{k^2}{n_y} + O\left(\frac{k^4}{n_y^2}\right),
\end{align*}
uniformly for $k=1:k_n$ as $n_x,n_y\to\infty$. In particular,
\begin{align*}
\mathbb{E}\left(\mathop{\max}\limits_{k=1:k_n}|\hat{\lambda}(k)-\lambda(k)|^2\right)&\le\mathop{\sum}\limits_{k=1}^{k_n}\mathbb{V}(\hat{\lambda}(k))\le\frac{k_n}{n_x}+\frac{k_n^3}{n_y} + O\left(\frac{k_n^5}{n_y^2}\right).
\end{align*}
Due to the definition of the coefficients $k_n$, the right-hand side above tends to zero, and (\ref{Eqn:ConsL2Small}) follows.
\end{proof}

\noindent\textbf{Proof of Theorem~\ref{Thm:ASConv}.} Note that $\theta(k)\!=\!\mathbb{E}(h(X_1,Y_1,\ldots,Y_k))$, with $h(x_1,y_1,\ldots,y_k)\!:=\![\![x_1\!\notin\!\{y_1,\ldots,y_k\}]\!]$. We show that the kernel function $h$ and the U-statistics $\hat{\theta}(k)$ satisfy the hypotheses of Lemma~\ref{Lem:L2Unif}. From this the theorem is immediate because $\mathcal{L}^2$-convergence implies convergence in probability.

Clearly $h$ is $(1,k)$-symmetric and $0\le h\le1$, which shows assumption (i) in Lemma~\ref{Lem:L2Unif}. On the other hand, due to the Law of Large Numbers, $\lim_{k\to\infty}h(X_1,Y_1,\ldots,Y_k)=[\![X_1\notin I_y]\!]$ almost surely, from which assumption (ii) in the lemma also follows.

Finally, to show assumption (iii), recall that $S_{k,n}$ is the set of one-to-one functions from $\{1,\ldots,k\}$ into $\{1,\ldots,n\}$; in particular, $|S_{k+1,n_y}|=(n_y-k)\cdot|S_{k,n_y}|$. Now note that for each indicator of the form $[\![X_1\notin\{Y_{\sigma(1)},\ldots,Y_{\sigma(k)}\}]\!]$, with $\sigma\in S_{k+1,n_y}$, there are $(n_y-k)$ choices of $\sigma(k+1)$  outside the set $\{\sigma(1),\ldots,\sigma(k)\}$. Because $[\![X_1\notin\{Y_{\sigma(1)},\ldots,Y_{\sigma(k)}\}]\!]\ge[\![X_1\notin\{Y_{\sigma(1)},\ldots,Y_{\sigma(k+1)}\}]\!]$, it follows that $\hat{\theta}(k)\ge\hat{\theta}(k+1)$ for all $k=1:(n_y-1)$. This shows condition (iii) in Lemma~\ref{Lem:L2Unif}, and Theorem~\ref{Thm:ASConv} follows.\hfill$\Box$\\

\noindent\textbf{Proof of equation~(\ref{Eqn:Vx}).} The jackknife estimate of the variance of $\hat\theta(k)$ obtained from removing a single $x$-data is, by definition, the quantity
\begin{align}
\label{Eqn:Vxu}
S^2_x(k) &:= \frac{n_x-1}{n_x}\sum_{i=1}^{n_x}\left(\frac{1}{(n_x-1)|S_{k,n_y}|}\sum\limits_{j\ne i}\sum\limits_{\sigma\in S_{k,n_y}}[\![X_j\notin\{Y_{\sigma(1)},\ldots,Y_{\sigma(k)}\}]\!]-\hat{\theta}(k)\right)^2.
\end{align}

Note that removing a color from the $x$-data which would otherwise add to $Q(j)$, decrements this quantity by one unit. Let $Q_i$ denote the $Q$-statistics associated with the data when observation $X_i$ from urn-$x$ is removed from the sample. Note that as each draw from urn-$x$ contributes to exactly one $Q(j)$, $Q_i(j) = Q(j)$ for all $j$ except for some $j_i^{\star}$ where $Q_i(j_i^{\star}) = Q(j_i^{\star})-1$. We have therefore that
\begin{align*}
S^2_x(k) &=\frac{n_x-1}{n_x}\sum_{i=1}^{n_x}\left(\sum\limits_{j=0}^{n_y-k}\frac{{n_y-j\choose k}Q_i(j)}{{(n_x-1)|S_{k,n_y}|}}-\sum\limits_{j=0}^{n_y-k}\frac{{n_y-j\choose k}Q(j)}{n_x|S_{k,n_y}|}\right)^2,\\
&=\frac{n_x-1}{n_x}\sum_{i=1}^{n_x}\left(\sum\limits_{j=0}^{n_y-k}\frac{{n_y-j\choose k}Q(j)}{n_x(n_x-1)|S_{k,n_y}|}-\frac{{n_y-j_i^{\star}\choose k}}{|S_{k,n_y}|(n_x-1)}\right)^2,\\
&=\frac{1}{n_x(n_x-1)}\sum_{i=1}^{n_x}\left(\hat{\theta}(k)-\frac{{n_y-j_i^{\star}\choose k}}{|S_{k,n_y}|}\right)^2.
\end{align*}
Since there are $Q(j)$ draws from urn-$x$ which contribute to $Q(j)$, the above sum may be now rewritten in the form given in (\ref{Eqn:Vx}).\hfill$\Box$\\

\noindent\textbf{Proof of equation~(\ref{Eqn:Vy}).} Similarly, $S^2_y(k)$ corresponds to the jackknife summed over each possible deletion of {a single $y$-data, which is more precisely given by
\begin{align}
\label{Eqn:Vyu}
S^2_y(k) &= \frac{n_y-1}{n_y}\sum_{r=1}^{n_y}\left(\frac{1}{n_x|S_r|}\sum\limits_{i=1}^{n_x}\sum\limits_{\sigma\in S_{r}}[\![X_i\notin\{Y_{\sigma(1)},\ldots,Y_{\sigma(k)}\}]\!]-\hat{\theta}(k)\right)^2,
\end{align}
where $S_{r}$ is the set of one-to-one functions from $\{1,\ldots,k\}$ into $\{1,\ldots,n_y\}\setminus\{r\}$.

Recall that $M(i,j)$ is the number of colors seen $i$ times in draws from urn-$x$ and $j$ times in draws from urn-$y$, giving that $\mathop{\sum}_iiM(i,j) = Q(j)$.

Fix $1\le r\le n_y$ and suppose that $Y_r$ is of a  color that contributes to $M(i_r^{\star},j_r^{\star})$, for some $i_r^{\star},j_r^{\star}.$ Removing $Y_r$ from the data decrements $M(i_r^{\star},j_r^{\star})$ and increments $M(i_r^{\star},j_r^{\star}-1)$ by one unit. Proceeding similarly as in the case for $S^2_x(k)$, if $M_r$ is used to denote the $M$-statistics when observation $Y_r$ is removed from sample-$y$, then
\begin{align*}
S^2_y(k)&=\frac{n_y-1}{n_y}\sum_{r=1}^{n_y}\left(\sum\limits_{j=0}^{n_y-k-1}\frac{{n_y-j-1\choose k}}{n_x|S_r|}\sum\limits_{i=1}^{n_x}iM_r(i,j)-\hat\theta(k)\right)^2,\\
&=\frac{n_y-1}{n_y}\sum_{r=1}^{n_y}\left(i_r^{\star}\frac{{n_y-j_r^{\star}\choose k}}{n_x|S_r|} -i_r^{\star}\frac{{n_y-j_r^{\star}-1\choose k}}{n_x|S_r|}+\sum_{j=0}^{n_y-k-1}\frac{{n_y-j_r^{\star}-1\choose k}}{n_x|S_r|}Q(j) -\hat{\theta}(k)\right)^2,\\
&=\frac{n_y-1}{n_y}\sum_{r=1}^{n_y}\left(i_r^{\star}\left(c_{j_r^{\star}-1}(k) - c_{j_r^{\star}}(k)\right)+\hat{\theta}_y(k) -\hat{\theta}(k)\right)^2,
\end{align*}
where $c_j(k)$ and $\hat{\theta}_y(k)$ are as defined in (\ref{Eqn:CjDef}) and (\ref{Eqn:ThetaYDef}). Noting that for each $i$ there are $j$ draws from urn-$y$ that contribute to $M(i,j)$, the form in (\ref{Eqn:Vy}) follows.\hfill$\Box$\\

In what follows, we specialize the coefficients in~(\ref{def:xi0jkGeneral}) and~(\ref{def:xi1jkGeneral}) to the kernel function of dissimilarity, $h(x_1,y_1,\ldots,y_k):=[\![x_1\notin\{y_1,\ldots,y_k]\!]$. From now on, for each $j\ge0$ and $k\ge1$, define
\begin{align}
\label{Eqn:Xi0j}
\xi_{0,j}(k)&:=\mathbb{V}(\Prob(X_1\notin\{Y_1,\ldots,Y_k\}|Y_1,\ldots,Y_j));\\
\label{Eqn:Xi1j}
\xi_{1,j}(k)&:=\mathbb{V}(\Prob(X_1\notin\{Y_1,\ldots,Y_k\}|X_1,Y_1,\ldots,Y_j)).
\end{align}
Above it is understood that the sigma-field generated by $(Y_1,\ldots,Y_j)$ when $j=0$ is $\{\emptyset,\Omega\}$; in particular, $\xi_{0,0}(k)=0$, for all $k\ge1$.

The following asymptotic properties of $\xi_{i,j}(k)$ are useful in the remaining proofs.

\begin{lem}
\label{Lem:Xi}
Assume that conditions (a)-(c) are satisfied and define $c:=\min_{i\in(I_x\cap I_y)}\Prob_y(i)$. It follows that $0<c<1$ and $0<\theta(\infty)<1$. Furthermore
\begin{align}
\label{Eqn:Xi2}
\xi_{1,k}(k)-\xi_{1,0}(k)&=\Theta\big((1-c)^k\big);\\
\label{Eqn:XiC2}
\xi_{1,0}(k) &= \theta(\infty)\,\big(1-\theta(\infty)\big)+\Theta\big((1-c)^k\big);\\
\label{Eqn:Xi1}
\xi_{0,k}(k) &= O\big((1-c)^k\big).
\end{align}
\end{lem}

\begin{proof}
Observe that conditions (a)-(b) imply that $0<c<1$. In addition, condition (b) implies that $\theta(\infty)<1$, whereas condition (c) implies that $\theta(\infty)>0$.

Next, consider the set
\[I^\star:=\big\{i\in(I_x\cap I_y)\hbox{ such that }\Prob_y(i)=c\big\},\]
i.e. $I^\star$ is the set of rarest colors in urn-$y$ which are also in urn-$x$.
Also note that
\begin{equation}
\label{ide:souseful}
\theta(k)=\theta(\infty)+\sum_{i\in(I_x\cap I_y)}\Prob_x(i)\big(1-\Prob_y(i)\big)^k.
\end{equation}

As an intermediate step before showing (\ref{Eqn:Xi2}), we prove that
\begin{equation}
\label{Eqn:XiB}
\xi_{1,j}(k)=\theta(2k-j)-\theta^2(k).
\end{equation}
For this, first observe that
\[\Prob(X_1\notin\{Y_1,\ldots,Y_k\}\mid X_1,Y_1,\ldots,Y_j)=[\![X_1\notin\{Y_1,\ldots,Y_j\}]\!](1-\Prob_y(X_1))^{k-j}.\]
Hence
\begin{align*}
\Esp\big(\Prob(X_1\notin\{Y_1,\ldots,Y_k\}\mid X_1,Y_1,\ldots,Y_j)^2\big)&=\Esp\big([\![X_1\notin\{Y_1,\ldots,Y_j\}]\!](1-\Prob_y(X_1))^{2k-2j}\big),\\
&=\Esp\big(\Prob(X_1\notin\{Y_1,\ldots,Y_{2k-j}\}\mid X_1,Y_1,\ldots,Y_j)\big),\\
&=\theta(2k-j),
\end{align*}
from which (\ref{Eqn:XiB}) now easily follows.

To show (\ref{Eqn:Xi2}) note that (\ref{Eqn:XiB}) implies 
\begin{align*}
\xi_{1,k}(k)-\xi_{1,0}(k)&=\theta(k)-\theta(2k),\\
&=\mathop{\sum}\limits_{i\in I_x}\Prob_x(i)(1-\Prob_y(i))^k(1-(1-\Prob_y(i))^k),\\
&=\mathop{\sum}\limits_{i\in (I_x\cap I_y)}\Prob_x(i)(1-\Prob_y(i))^k(1-(1-\Prob_y(i))^k),\\
&=(1-c)^k\mathop{\sum}\limits_{i\in I^{\star}}\Prob_x(i) + o\left((1-c)^k\right),
\end{align*}
which establishes (\ref{Eqn:Xi2}).

Now note that
\begin{eqnarray*}
\xi_{1,0}(k)&=&\theta(2k)-\theta^2(k),\\
&=&\mathop{\sum}\limits_{i\in I_x}\Prob_x(i)\big(1-\Prob_y(i)\big)^{2k} - \left(\mathop{\sum}\limits_{i\in I_x}\Prob_x(i)\big(1-\Prob_y(i)\big)^k\right)^2,\\
&=&\theta(\infty)-\theta^2(\infty) -2\theta(\infty)(1-c)^k\left(\mathop{\sum}\limits_{i\in I^{\star}}\Prob_x(i)\right)+o(1-c)^k,
\end{eqnarray*}
which establishes (\ref{Eqn:XiC2}).

Next we show (\ref{Eqn:Xi1}), which we note gives more precise information than (\ref{ineq:Var21}). Consider the random variable $T$ defined as the smallest $n\ge1$ such that $(I_x\cap I_y)\subset\{Y_1,\ldots,Y_n\}$. We may bound the probability of $T$ being large by  $\Prob(T> k)\le n(1-c)^{k}$, where $n:=|I_x\cap I_y|$ is finite because of condition (a). On the other hand, note that
\[\Prob\big(X_1\notin\{Y_1,\ldots,Y_k\}\mid Y_1,\ldots,Y_k\big)=1-\Prob_x(\{Y_1,\ldots,Y_k\}).\]
Define $W_k:=1-\Prob_x(\{Y_1,\ldots,Y_k\})-\theta(k)$ and observe that, over the event $T\le k$, $W_k=\theta(\infty)-\theta(k)$. Since $|W_k|\le1$, we obtain that
\begin{eqnarray*}
\xi_{0,k}(k) &=& \mathbb{E}\Big(\mathbb{E}\big( W_k^2\mid T>k\big)\Big)\cdot\Prob(T>k)+\mathbb{E}\Big(\mathbb{E}\big( W_k^2\mid T\le k\big)\Big)\cdot\Prob(T\le k),\\
&\le& \Prob(T>k)+\mathbb{E}\Big(\mathbb{E}\big( W_k^2\mid T\le k\big)\Big),\\
&\le& n(1-c)^k+\big(\theta(\infty)-\theta(k)\big)^2.
\end{eqnarray*}
The identity in equation (\ref{Eqn:Xi1}) is now a direct consequence of (\ref{ide:souseful}).
\end{proof}

Our next goal is to show Theorems~\ref{Thm:VarAcc} and~\ref{Thm:CLT}. To do so we rely on the method of projection by Grams and Serfling~\cite{GramsSerfling}. This approach approximates $\hat{\theta}(k)$ by the random variable
\begin{align*}
\hat{\theta}_P(k)&:=\theta(k) + \mathop{\sum}\limits_{i=1}^{n_x}(\Esp(\hat{\theta}(k)|X_i)-\theta(k)) + \mathop{\sum}\limits_{j=1}^{n_y}(\Esp(\hat{\theta}(k)|Y_j)-\theta(k)).
\end{align*}
The projection is the best approximation in terms of mean squared error to $\hat{\theta}(k)$ that is a linear combination of individual functions of each datapoint.

Under the stated conditions, $\hat{\theta}_P(k)$ is the sum of two independent sums of non-degenerate i.i.d. random variables and therefore satisfies the hypotheses of the classical central limit theorem. The variance of the projection is easier to analyze and estimate than the $U$-statistic directly, which is relevant in establishing consistency for the jackknife estimation of variance.

Let
\begin{align*}
R(k)&:=\hat{\theta}(k)-\hat{\theta}_P(k),
\end{align*}
be the remainder of $\hat{\theta}(k)$ that is not accounted for by its projection. When $R(k)$ is small relative to $\hat{\theta}_P(k)$, $\hat{\theta}(k)$ is mostly explained by $\hat{\theta}_P(k)$ in relative terms.

The next lemma summarizes results about the asymptotic properties of $R(k)$, particularly with relation to the scale of $\hat{\theta}_P(k)$ as given by its variance.

\begin{lem}
\label{Lem:Rk}
We have that
\begin{align}
\label{Eqn:Rk1}
\mathbb{V}(\hat{\theta}_P(k))&=n_x^{-1}\xi_{1,0}(k)+k^2 n_y^{-1}\xi_{0,1}(k),\\
\label{ide:ER2k explicit}
\Esp(R^2(k))&=\mathbb{V}(\hat{\theta}(k))-\mathbb{V}(\hat{\theta}_P(k)).
\end{align}
Under assumptions (a)-(c), for a fixed $k\ge1$, we have that 
\begin{align}
\label{Eqn:ForJackknifeFixedk}
\mathop{\lim}\limits_{n_x,n_y\rightarrow\infty}&\left|\frac{\mathbb{V}(\hat{\theta}(k))}{\mathbb{V}(\hat{\theta}_P(k))}-1\right|=0.
\end{align}
Furthermore, under assumptions (a)-(d) we have that
\begin{align}
\label{Eqn:RkUnifConv}
\mathop{\lim}\limits_{n_x,n_y\rightarrow\infty}\max_{k=1:n_y}&\left|\frac{\mathbb{V}(\hat{\theta}(k))}{\mathbb{V}(\hat{\theta}_P(k))}-1\right|=0;\\
\label{Eqn:UnifProbRemConv}
\mathop{\lim}\limits_{n_x,n_y\rightarrow\infty}\max_{k=1:n_y}&\mathbb{P}\left(\left|\frac{\hat{\theta}(k)-\hat{\theta}_P(k)}{\sqrt{\mathbb{V}(\hat{\theta}_P(k))}}\right|>\epsilon\right)=0;
\end{align}
for all $\epsilon>0$.
\end{lem}

\begin{proof}
A direct calculation from the form given in (\ref{Eqn:UStatDef}) gives that
\begin{align}
\label{Eqn:VarHatX}
\Esp(\hat{\theta}(k)|X_i))&=n_x^{-1}\Prob(X_i\notin\{Y_1,\ldots,Y_k\}|X_i) + \left(1-n_x^{-1}\right)\theta(k);\\
\nonumber
\mathbb{V}(\Esp(\hat{\theta}(k)|X_i))&=\frac{\xi_{1,0}(k)}{n_x^2};\\
\label{Eqn:VarHatY}
\Esp(\hat{\theta}(k)|Y_i))&=k n_y^{-1}\Prob(X_i\notin\{Y_i,\ldots,Y_{k+i-1}\}|Y_i) + \left(1-k n_y^{-1}\right)\theta(k);\\
\nonumber
\mathbb{V}(\Esp(\hat{\theta}(k)|Y_i))&=\frac{k^2\xi_{0,1}(k)}{n_y^2}.
\end{align}
As $\mathbb{V}(\hat{\theta}_P(k)) = n_x\mathbb{V}(\Esp(\hat{\theta}(k)|X_i)) + n_y\mathbb{V}(\Esp(\hat{\theta}(k)|Y_i))$, (\ref{Eqn:Rk1}) follows. 

To show (\ref{ide:ER2k explicit}), first observe that
\begin{equation}
\label{ide:ER2ka}
\mathbb{E}(R^2(k))=\mathbb{V}(R(k))=\mathbb{V}(\hat\theta(k))+\mathbb{V}(\hat\theta_P(k))-2\,\mbox{Cov}(\hat{\theta}(k),\hat{\theta}_P(k)).
\end{equation}
Next, using the definition of the projection, we obtain that
\begin{eqnarray*}
\mbox{Cov}(\hat{\theta}(k),\hat{\theta}_P(k))&=&\sum_{i=1}^{n_x}\mbox{Cov}(\hat{\theta}(k),\mathbb{E}(\hat\theta(k)|X_i))+\sum_{j=1}^{n_y}\mbox{Cov}(\hat{\theta}(k),\mathbb{E}(\hat\theta(k)|Y_j)),\\
&=&\sum_{i=1}^{n_x}\mathbb{V}(\mathbb{E}(\hat\theta(k)|X_i))+\sum_{j=1}^{n_y}\mathbb{V}(\mathbb{E}(\hat\theta(k)|Y_j)),\\
&=&\mathbb{V}\left(\sum_{i=1}^{n_x}\mathbb{E}(\hat\theta(k)|X_i)+\sum_{j=1}^{n_y}\mathbb{E}(\hat\theta(k)|Y_j)\right),\\
&=&\mathbb{V}(\hat\theta_P(k)),
\end{eqnarray*}
from which (\ref{ide:ER2k explicit}) follows, due to the identity in (\ref{ide:ER2ka}). Note that the last identity implies that $\hat\theta_P(k)$ and $R(k)$ are uncorrelated.

Before continuing, we note that (\ref{Eqn:UnifProbRemConv}) is a direct consequence of (\ref{ide:ER2k explicit}), (\ref{Eqn:RkUnifConv}) and Chebyshev's inequality~\cite{Durrett}. To complete the proof of the lemma all reduces therefore to show (\ref{Eqn:UnifProbRemConv}) under conditions (a)-(d). Indeed, if $b>1$ and we let $k_n=\log_b(n_y)$ then due to the identities in (\ref{Eqn:LambdaVar}) and (\ref{Eqn:Rk1}) and Lemma~\ref{Lem:ChooseRelation}, we obtain under (a)-(c) that
\[\mathbb{V}(\hat\theta(k))=\mathbb{V}(\hat\theta_P(k))+O\left(\frac{k^2}{n_xn_y}+\frac{k^4}{n_y^2}\right),\]
uniformly for all $k=1:k_n$, as $n_x,n_y\to\infty$. Since $\mathbb{V}(\hat\theta_P(k))>0$ for all $k\ge1$, we have thus shown (\ref{Eqn:ForJackknifeFixedk}). Furthermore, note that $\xi_{1,0}(k)>0$ for all $k\ge1$; in particular, due to (\ref{Eqn:XiC2}) and conditions (a)-(d), we can assert that $\inf_{k\ge1}\xi_{1,0}(k)>0$. Since $\xi_{0,1}(k)\ge0$, the above identity together with the one in (\ref{Eqn:Rk1}) let us conclude that \[\max_{k=1:k_n}\left|\frac{\mathbb{V}(\hat\theta(k))}{\mathbb{V}(\hat\theta_P(k))}-1\right|=O\left(\frac{k_n^2}{n_y}+\frac{n_x k_n^4}{n_y^2}\right),\]
as $n_x,n_y\to\infty$. Because of condition (d), the big-O term above tends to $0$. As a result:
\begin{equation}
\label{ide:primera}
\lim_{n_x,n_y\to\infty}\,\max_{k=1:k_n}\left|\frac{\mathbb{V}(\hat\theta(k))}{\mathbb{V}(\hat\theta_P(k))}-1\right|=0.
\end{equation}

On the other hand, (\ref{ide:ER2k explicit}) implies that $\mathbb{V}(\hat\theta(k))\ge\mathbb{V}(\hat\theta_P(k))$. Hence, using (\ref{ide:CombIdent}) and (\ref{ide:sosouseful}) to bound from above the variance of the U-statistic, we obtain:
\[1\le\frac{\mathbb{V}(\hat\theta(k))}{\mathbb{V}(\hat\theta_P(k))}\le\frac{\xi_{0,k}(k)+\xi_{1,k}(k)/n_x}{k^2\xi_{0,1}(k)/n_y+\xi_{1,0}(k)/n_x}\le n_x\,\frac{\xi_{0,k}(k)}{\xi_{1,0}(k)}+\frac{\xi_{1,k}(k)}{\xi_{1,0}(k)}=1+n_x\cdot O\big((1-c)^k\big),\]
as $k\to\infty$, where for the last identity we have used (\ref{Eqn:Xi2}) and (\ref{Eqn:Xi1}). Since $n_x=\Theta(n_y)$, it follows from the above identity that
\[\max_{k=k_n:n_y}\left|\frac{\mathbb{V}(\hat\theta(k))}{\mathbb{V}(\hat\theta_P(k))}-1\right|=O\big(n_y(1-c)^{k_n}\big)=O\big(n_y^{1+\log_b(1-c)}\big).\]
In particular, if the base-$b$ in the logarithm is selected to satisfy that $1<b<1/(1-c)$, then
\begin{equation}
\label{ide:segunda}
\lim_{n_x,n_y\to\infty}\,\max_{k=k_n:n_y}\left|\frac{\mathbb{V}(\hat\theta(k))}{\mathbb{V}(\hat\theta_P(k))}-1\right|=0.
\end{equation}
The identities in equation (\ref{ide:primera}) and (\ref{ide:segunda}) show (\ref{Eqn:UnifProbRemConv}), which completes the proof of the lemma.
\end{proof}

\noindent\textbf{Proof of Theorem~\ref{Thm:CLT}.} For a fixed $k$, note that $\hat\theta_P(k)$ is the sum of two independent sums of non-degenerate i.i.d. random variables and thus,
\[(\hat\theta_P(k)-\theta(k))/\sqrt{\mathbb{V}(\hat\theta_P(k))}\]
is asymptotically a standard Normal random variable as $n_x,n_y\to\infty$ by the classical Central Limit Theorem. We would like to show however that this convergence also applies if we let $k$ vary with $n_x$ and $n_y$. We do so using the Berry-Esseen inequality~\cite{Shevtsova}. Motivated by this we define the random variables
\begin{align*}
X_i^{\prime}(k)&:= \frac{\Esp(\hat{\theta}(k)|X_i)-\theta(k)}{\sqrt{\mathbb{V}(\hat{\theta}_P(k))}};\\
Y_j^{\prime}(k)&:=\frac{\Esp(\hat{\theta}(k)|Y_j)-\theta(k)}{\sqrt{\mathbb{V}(\hat{\theta}_P(k))}}.
\end{align*}
Note that $\Esp(X_i^{\prime}(k)) = \Esp(Y_j^{\prime}(k)) = 0$, and that 
\begin{align*}
\mathop{\sum}\limits_{i=1}^{n_x}\Esp(|X_i^{\prime}(k)|^2) + \mathop{\sum}\limits_{j=1}^{n_y}\Esp(|Y_j^{\prime}(k)|^2) = 1.
\end{align*}
We need to show that
\begin{equation}
\label{ide:InterBerryEss}
\mathop{\sum}\limits_{i=1}^{n_x}\Esp(|X_i^{\prime}(k)|^3) + \mathop{\sum}\limits_{j=1}^{n_y}\Esp(|Y_j^{\prime}(k)|^3) = o(1),
\end{equation}
uniformly for $k=1:n_y$, as $n_x,n_y\to\infty$.

Note that from (\ref{Eqn:VarHatX}) and (\ref{Eqn:VarHatY}),
\begin{align*}
\big|\Esp(\hat{\theta}(k)|X_i)-\theta(k)\big|^3&=\frac{|\Prob(X_i\notin\{Y_1,\ldots,Y_k\}|X_i)-\theta(k)|^3}{n^3_x};\\
\big|\Esp(\hat{\theta}(k)|Y_i)-\theta(k)\big|^3&=\frac{k^3|\Prob(X_1\notin\{Y_i,\ldots,Y_{k+1-i}\}|Y_i)-\theta(k)|^3}{n^3_y}.
\end{align*}
Let
\begin{align*}
\eta_{1,0}(k)&:=\Esp|\Prob(X_i\notin\{Y_1,\ldots,Y_k\}|X_i)-\theta(k)|^3;\\
\eta_{0,1}(k)&:=\Esp|\Prob(X_1\notin\{Y_i,\ldots,Y_{k+1-i}\}|Y_i)-\theta(k)|^3.
\end{align*}
It follows from (\ref{Eqn:Rk1}) that
\begin{align*}
\mathop{\sum}\limits_{i=1}^{n_x}\Esp(|X_i^{\prime}(k)|^3) + \mathop{\sum}\limits_{j=1}^{n_y}\Esp(|Y_j^{\prime}(k)|^3) &= \frac{\eta_{1,0}(k)/n^2_x+k^3\eta_{0,1}(k)/n^2_y}{\big(\mathbb{V}(\hat{\theta}_P(k))\big)^{3/2}}= \frac{\eta_{1,0}(k)/n^2_x+k^3\eta_{0,1}(k)/n^2_y}{\big(\xi_{1,0}(k)/n_x+k^2\xi_{0,1}(k)/n_y\big)^{3/2}}.
\end{align*}
But note that $0\le\eta_{0,1}(k)\le\xi_{0,1}(k)$. Since, according to Lemma~\ref{Lem:Xi}, $\xi_{0,1}(k)$ decreases exponentially fast, we obtain
\begin{align*}
k^3\eta_{0,1}(k)/n_y^2&=O(n_y^{-2}),
\end{align*}
uniformly for all $k=1:n_y$, as $n_y\rightarrow\infty$. On the other hand, $0\le\eta_{1,0}(k)\le\xi_{1,0}(k)\le1$. Furthermore, (\ref{Eqn:XiC2}) implies that $\inf_{k\ge1}\xi_{1,0}(k)>0$. Since $n_x = \Theta(n_y)$, for some finite constant $C>0$ we find that
\begin{align*}
\mathop{\sum}\limits_{i=1}^{n_x}\Esp(|X_i^{\prime}(k)|^3) + \mathop{\sum}\limits_{j=1}^{n_y}\Esp(|Y_j^{\prime}(k)|^3) &\le C\frac{1/n^2_x+ 1/n^2_y}{\left(\mathop{\inf}\limits_{k\ge1}\xi_{1,0}(k)/n_x\right)^{3/2}}=O\left(\frac{1}{\sqrt{n_x}}\right),
\end{align*}
which shows (\ref{ide:InterBerryEss}).

The above establishes convergence in distribution of $(\hat{\theta}_P(k)-\theta(k))/\sqrt{\mathbb{V}(\hat{\theta}_P(k))}$ to a standard normal random variable uniformly for $k=1:n_y$, as $n_x,n_y\to\infty$. The end of the proof is an adaptation of the proof of Slutsky's Theorem~\cite{Slutsky}. Indeed, note that
\begin{align}
\label{Int Slutzky}
\Prob\left(\frac{\hat{\theta}(k)-\theta(k)}{\sqrt{\mathbb{V}(\hat{\theta}(k))}}\le t\right)=\Prob\left(\frac{\hat{\theta}_P(k)-\theta(k)}{\sqrt{\mathbb{V}(\hat{\theta}_P(k))}}+\frac{\hat{\theta}(k)-\hat{\theta}_P(k)}{\sqrt{\mathbb{V}(\hat{\theta}_P(k))}}\le t\sqrt{\frac{\mathbb{V}(\hat{\theta}(k))}{\mathbb{V}(\hat{\theta}_P(k))}}\right).
\end{align}
From this identity, it follows for any fixed $\epsilon>0$ that
\[\Prob\left(\frac{\hat{\theta}(k)-\theta(k)}{\sqrt{\mathbb{V}(\hat{\theta}(k))}}\le t\right)\le \Prob\left(\frac{\hat{\theta}_P(k)-\theta(k)}{\sqrt{\mathbb{V}(\hat{\theta}_P(k))}}\le t\sqrt{\frac{\mathbb{V}(\hat{\theta}(k))}{\mathbb{V}(\hat{\theta}_P(k))}}+\epsilon\right)+\Prob\left(\left|\frac{\hat{\theta}(k)-\hat{\theta}_P(k)}{\sqrt{\mathbb{V}(\hat{\theta}_P(k))}}\right|\ge\epsilon\right).\]
The first term on the right-hand side of the above inequality can be made as close to $\Prob[Z\le t+\epsilon]$ as wanted, uniformly for $k=1:n_y$, as $n_x,n_y\to\infty$, because of (\ref{Eqn:RkUnifConv}). On the other hand, the second term tends to $0$ uniformly for $k=1:n_y$ because of (\ref{Eqn:UnifProbRemConv}). Letting $\epsilon\to0^+$, shows that
\[\limsup_{n_x,n_y\to\infty}\,\max_{k=1:n_y}\Prob\left(\frac{\hat{\theta}(k)-\theta(k)}{\sqrt{\mathbb{V}(\hat{\theta}(k))}}\le t\right)\le\Prob[Z\le t].\]
Similarly, using (\ref{Int Slutzky}), we have:
\[\Prob\left(\frac{\hat{\theta}(k)-\theta(k)}{\sqrt{\mathbb{V}(\hat{\theta}(k))}}\le t\right)\ge \Prob\left(\frac{\hat{\theta}_P(k)-\theta(k)}{\sqrt{\mathbb{V}(\hat{\theta}_P(k))}}\le t\sqrt{\frac{\mathbb{V}(\hat{\theta}(k))}{\mathbb{V}(\hat{\theta}_P(k))}}-\epsilon\right)-\Prob\left(\left|\frac{\hat{\theta}(k)-\hat{\theta}_P(k)}{\sqrt{\mathbb{V}(\hat{\theta}_P(k))}}\right|\ge\epsilon\right),\]
and a similar argument as before shows now that
\[\liminf_{n_x,n_y\to\infty}\,\max_{k=1:n_y}\Prob\left(\frac{\hat{\theta}(k)-\theta(k)}{\sqrt{\mathbb{V}(\hat{\theta}(k))}}\le t\right)\ge\Prob[Z\le t],\]
which completes the proof of the theorem.\hfill$\Box$\\

We finally show Theorem~\ref{Thm:VarAcc}, for which we first show the following result.
\begin{lem}
\label{Lem:JackknifeAcc}
Let $S^i_{k,n}$ be the set of one-to-one functions from $\{1,\ldots,k\}$ into $\{1,\ldots,n\}/\{i\}$. Consider the kernel $h(x_1,y_1,\ldots,y_k):=[\![x_1\notin\{y_1,\ldots,y_k\}]\!]$, and define
\begin{align}
\label{Eqn:JKKernStart}
\hat{\theta}_x^{i}(k)&:=\frac{1}{|S_{k,n_y}|}\mathop{\sum}\limits_{\sigma\in S_{k,n_y}}\frac{1}{n_x-1}\mathop{\sum}\limits_{j=1,j\ne i}^{n_x}h(X_j,Y_{\sigma(1)},\ldots,Y_{\sigma(k)});\\
\hat{\theta}_x^{i\prime}(k)&:=\frac{1}{|S_{k,n_y}|}\mathop{\sum}\limits_{\sigma\in S_{k,n_y}}h(X_i,Y_{\sigma(1)},\ldots,Y_{\sigma(k)});\\
\label{Eqn:JKKernNo2}
\hat{\theta}_y^{i}(k)&:=\frac{1}{|S^i_{k,n_y}|}\mathop{\sum}\limits_{\sigma\in S^i_{k,n_y}}\frac{1}{n_x}\mathop{\sum}\limits_{j=1}^{n_x}h(X_j,Y_{\sigma(1)},\ldots,Y_{\sigma(k)});\\
\label{Eqn:JKKernEnd}
\hat{\theta}_y^{i\prime}(k)&:=\frac{1}{|S^i_{k-1,n_y}|}\mathop{\sum}\limits_{\sigma\in S^i_{k-1,n_y}}\frac{1}{n_x}\mathop{\sum}\limits_{j=1}^{n_x}h(X_j,Y_i,Y_{\sigma(1)},\ldots,Y_{\sigma(k-1)}).
\end{align}
Then, for each $k\ge1$ and $\epsilon>0$,
\begin{align}
\label{Eqn:JKConsX}
\mathop{\lim}\limits_{n_x,n_y\to\infty}\Prob\left(\left|\mathop{\sum}\limits_{i=1}^{n_x}\frac{\left(\hat{\theta}_x^i(k)-\hat{\theta}_x^{i\prime}(k)\right)^2}{n_x}-\xi_{1,0}(k)\right|>\epsilon\right)&=0;\\
\label{Eqn:JKConsY}
\mathop{\lim}\limits_{n_x,n_y\to\infty}\Prob\left(\left|\mathop{\sum}\limits_{i=1}^{n_y}\frac{\left(\hat{\theta}_y^i(k)-\hat{\theta}_y^{i\prime}(k)\right)^2}{n_y}-\xi_{0,1}(k)\right|>\epsilon\right)&=0.
\end{align}
\end{lem}

\begin{proof}
Fix $k\ge1$. We first use a result by Sen~\cite{Sen} to show that, for each $i\ge1$:
\begin{align}
\label{Eqn:KernelX1}
\mathop{\lim}\limits_{n_x,n_y\rightarrow\infty}\hat{\theta}_x^i(k)&=\theta(k);\\
\label{Eqn:KernelX2}
\mathop{\lim}\limits_{n_x,n_y\rightarrow\infty}\hat{\theta}_x^{i\prime}(k)&=\mathbb{E}(h(X_i,Y_1,\ldots,Y_k)|X_i);\\
\label{Eqn:KernelY1}
\mathop{\lim}\limits_{n_x,n_y\rightarrow\infty}\hat{\theta}_y^i(k)&=\theta(k);\\
\label{Eqn:KernelY2}
\mathop{\lim}\limits_{n_x,n_y\rightarrow\infty}\hat{\theta}_y^{i\prime}(k)&=\mathbb{E}(h(X_1,Y_i,\ldots,Y_{i+k-1})|Y_i);
\end{align}
in an almost sure sense. Indeed, assume without loss of generality that $i=1$. As the kernel functions found in (\ref{Eqn:JKKernStart}) and (\ref{Eqn:JKKernNo2}) are bounded, the hypotheses of Theorem 1 in~\cite{Sen} are satisfied, from which (\ref{Eqn:KernelX1}) and (\ref{Eqn:KernelY1}) are immediate. Similarly, because $X_1$ and $Y_1$ are discrete random variables, (\ref{Eqn:KernelX2}) and (\ref{Eqn:KernelY2}) also follow from~\cite{Sen} .

Define
\begin{align}
\label{Eqn:JKUK}
U(k)&:=\mathop{\sum}\limits_{i=1}^{n_x}\frac{(\hat{\theta}_x^i(k)-\hat{\theta}_x^{i\prime}(k))^2}{n_x};\\
\label{Eqn:JKVK}
V(k)&:=\mathop{\sum}\limits_{i=1}^{n_y}\frac{(\hat{\theta}_y^i(k)-\hat{\theta}_y^{i\prime}(k))^2}{n_y};
\end{align} 
and observe that
\begin{align*}
\mathbb{V}\left(U(k)\right)&=\frac{\mathbb{V}\left((\hat{\theta}_x^1(k)-\hat{\theta}_x^{1\prime}(k))^2\right)}{n_x}+\frac{n_x-1}{n_x}\hbox{Cov}\left((\hat{\theta}_x^1(k)-\hat{\theta}_x^{1\prime}(k))^2,(\hat{\theta}_x^2(k)-\hat{\theta}_x^{2\prime}(k))^2\right);\\
\mathbb{V}\left(V(k)\right)&=\frac{\mathbb{V}\left((\hat{\theta}_y^1(k)-\hat{\theta}_y^{1\prime}(k))^2\right)}{n_y}+\frac{n_y-1}{n_y}\hbox{Cov}\left((\hat{\theta}_y^1(k)-\hat{\theta}_y^{1\prime}(k))^2,(\hat{\theta}_y^2(k)-\hat{\theta}_y^{2\prime}(k))^2\right).
\end{align*}
Furthermore, due to (\ref{Eqn:KernelX1})-(\ref{Eqn:KernelY2}), we have that
\begin{eqnarray}
\label{Eqn:ToFixCite1}
\lim_{n_x,n_y\to\infty}(\hat{\theta}_x^i(k)-\hat{\theta}_x^{i\prime}(k))^2&=&(\theta(k)-\mathbb{E}(h(X_i,Y_1,\ldots,Y_k)|X_i))^2;\\
\label{Eqn:ToFixCite2}
\lim_{n_x,n_y\to\infty}(\hat{\theta}_y^i(k)-\hat{\theta}_y^{i\prime}(k))^2&=&(\theta(k)-\mathbb{E}(h(X_1,Y_i,\ldots,Y_{i+k-1})|Y_i))^2.
\end{eqnarray}
But note that, for $i\ne j$, $(\theta(k)-\mathbb{E}(h(X_i,Y_1,\ldots,Y_k)|X_i))^2$ and $(\theta(k)-\mathbb{E}(h(X_j,Y_1,\ldots,Y_k)|X_j))^2$ are independent and hence uncorrelated. Similarly, the random variables $(\theta(k)-\mathbb{E}(h(X_1,Y_i,\ldots,Y_{i+k-1})|Y_i))^2$ and $(\theta(k)-\mathbb{E}(h(X_1,Y_j,\ldots,Y_{j+k-1})|Y_j))^2$ are independent. Since $|\hat{\theta}_x^i(k)-\hat{\theta}_x^{i\prime}(k)|\le 1$ and $|\hat{\theta}_y^i(k)-\hat{\theta}_y^{i\prime}(k)|\le 1$, it follows from (\ref{Eqn:ToFixCite1}) and (\ref{Eqn:ToFixCite2}), and the Bounded Convergence Theorem~\cite{Durrett} that
\begin{align}
\label{Eqn:ForDurrX}
\mathbb{V}\left(U(k)\right)&=o(1);\\
\label{Eqn:ForDurrY}
\mathbb{V}\left(V(k)\right)&=o(1);
\end{align}
as $n_x,n_y\to\infty$.

Finally, by (\ref{Eqn:Xi0j}) and (\ref{Eqn:Xi1j}) it follows that 
\begin{align*}
\xi_{1,0}(k)&=\Esp\left(\theta(k)-\mathbb{E}(h(X_1,Y_1,\ldots,Y_k)|X_1)\right)^2;\\
\xi_{0,1}(k)&=\Esp\left(\theta(k)-\mathbb{E}(h(X_1,Y_1,\ldots,Y_k)|Y_1)\right)^2.
\end{align*}
In particular, again by the Bounded Convergence Theorem, we have that $\lim_{n_x,n_y\to\infty}\mathbb{E}\left(U(k)\right)=\xi_{1,0}(k)$ and $\lim_{n_x,n_y\to\infty}\mathbb{E}\left(V(k)\right)=\xi_{0,1}(k)$. Since
\begin{align*}
U(k)-\xi_{1,0}(k)&=\big(\mathbb{E}(U(k))-\xi_{1,0}(k)\big)+\big(U(k)-\mathbb{E}(U(k))\big);\\
V(k)-\xi_{0,1}(k)&=\big(\mathbb{E}(V(k))-\xi_{0,1}(k)\big)+\big(V(k)-\mathbb{E}(V(k))\big);
\end{align*}
the lemma is now a direct consequence of (\ref{Eqn:ForDurrX}) and (\ref{Eqn:ForDurrY}), and Theorem 1.5.4 of Durrett~\cite{Durrett}.
\end{proof}

\noindent\textbf{Proof of Theorem~\ref{Thm:VarAcc}.}
Fix $k\ge1$. Using (\ref{Eqn:UStatDef}) we have that
\begin{align*}
\hat{\theta}(k)&=\left(1-\frac{1}{n_x}\right)\hat{\theta}_x^{i}(k) + \frac{1}{n_x}\hat{\theta}_x^{i\prime}(k);\\
\hat{\theta}_x^i(k)-\hat{\theta}(k)&=\frac{1}{n_x}\left(\hat{\theta}_x^i(k)-\hat{\theta}_x^{i\prime}(k)\right);\\
\hat{\theta}(k)&=\left(1-\frac{k}{n_y}\right)\hat{\theta}_y^{i}(k) + \frac{k}{n_y}\hat{\theta}_y^{i\prime}(k);\\
\hat{\theta}_y^i(k)-\hat{\theta}(k)&=\frac{k}{n_y}\left(\hat{\theta}_y^i(k)-\hat{\theta}_y^{i\prime}(k)\right).
\end{align*}

It follows by (\ref{Eqn:Vxu}) and (\ref{Eqn:Vyu}) that
\begin{align}
\label{Eqn:VarProbSx}
S^2_x(k)&=\frac{n_x-1}{n_x}\cdot\frac{U(k)}{n_x},\\
\label{Eqn:VarProbSy}
S^2_y(k)&=\frac{n_y-1}{n_y}\cdot\frac{k^2V(k)}{n_y},
\end{align}
where $U(k)$ and $V(k)$ are as in (\ref{Eqn:JKUK}) and (\ref{Eqn:JKVK}), respectively. Furthermore, observe that
\begin{align*}
S^2(k)&=S_x^2(k)+S_y^2(k)=\frac{n_x-1}{n_x}\cdot\frac{U(k)}{n_x} + \frac{n_y-1}{n_y}\cdot\frac{k^2V(k)}{n_y}.
\end{align*}
In particular, due to (\ref{Eqn:Rk1}), we obtain that
\[\left|\frac{S^2(k)}{\mathbb{V}(\hat\theta_P(k))}-1\right|\le\frac{|U(k)-\xi_{1,0}(k)|}{\xi_{1,0}(k)}+\frac{|V(k)-\xi_{0,1}(k)|}{\xi_{0,1}(k)}+\frac{|U(k)|}{n_x\,\xi_{1,0}(k)}+\frac{|V(k)|}{n_y\,\xi_{0,1}(k)}.\]
By Lemma~\ref{Lem:JackknifeAcc}, $U(k)$ converges in probability to $\xi_{1,0}(k)$, while similarly $V(k)$ converges in probability to $\xi_{0,1}(k)$; in particular, the first two terms on the right-hand side of the inequality converge to $0$ in probability. Since $|U(k)|\le1$ and $|V(k)|\le1$, the same can be said about the last two terms of the inequality. Consequently, $S^2(k)/\mathbb{V}(\hat\theta_P(k))$ converges to $1$ in probability, as $n_x,n_y\to\infty$. As stated in (\ref{Eqn:ForJackknifeFixedk}), however, conditions (a)-(c) imply that  $\mathbb{V}(\hat{\theta}_P(k))$ and $\mathbb{V}(\hat{\theta}(k))$ are asymptotically equivalent as $n_x,n_y\to\infty$, from which the theorem follows.\hfill$\Box$

\section*{Acknowledgments}
We thank Rob Knight for insightful discussions and comments about this manuscript, and Antonio Gonzalez for providing processed OTU tables from the Human Microbiome Project.

\bibliography{BibTexData}

\section*{Figure Legends}

\begin{figure}[ht!]
\begin{center}
\includegraphics[scale=0.93]{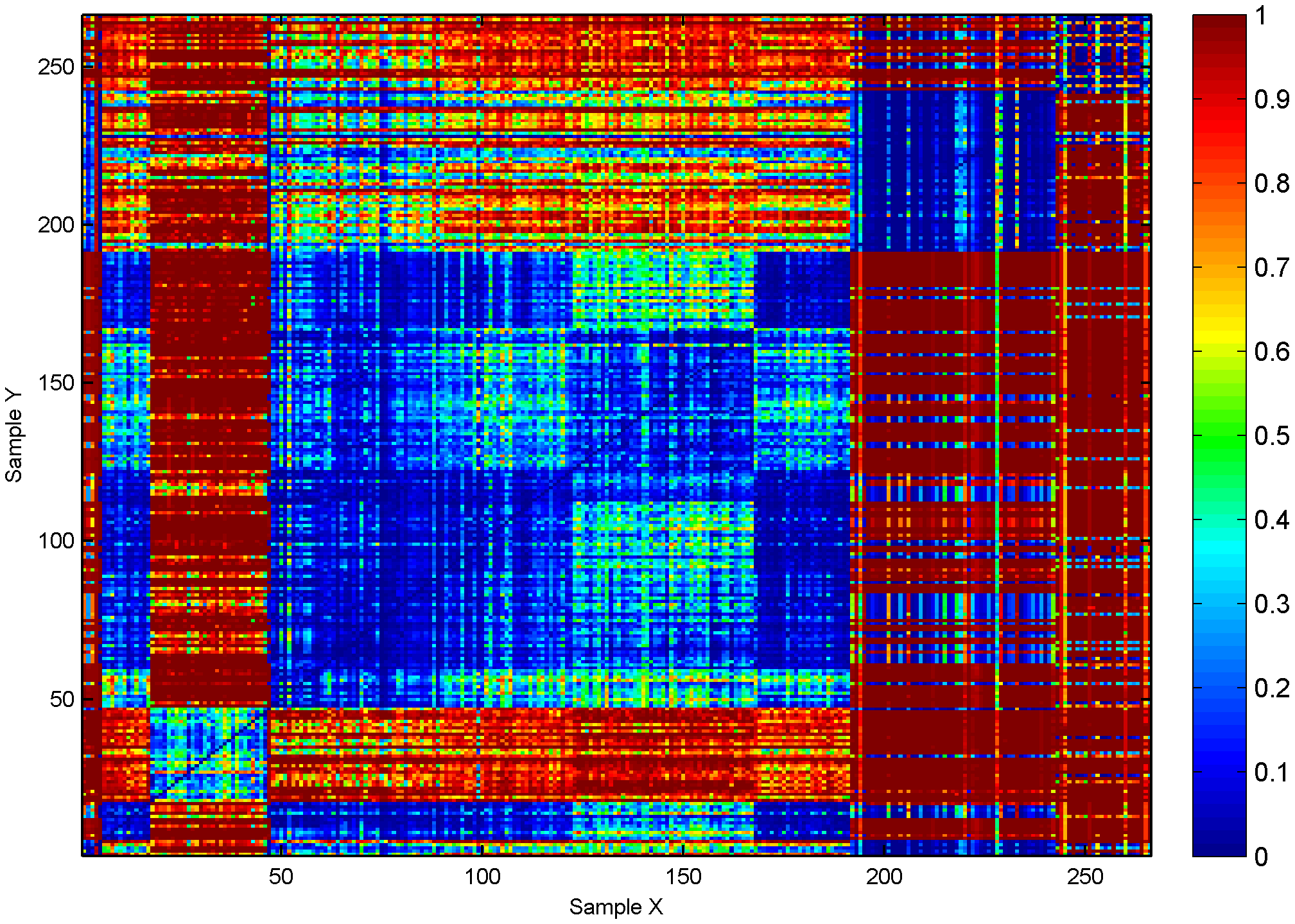}
\end{center}
\caption{{\bf Dissimilarity estimates.} Heat map of $\hat{\theta}(n_y)$ sorted by site location metadata. Here, the $x$-axis denotes the sample from the environment corresponding to urn-$x$, and similarly for the $y$-axis. The entries on the diagonal are set to zero.}
\label{Fig:ThetaMap}
\end{figure}

\begin{figure}[ht!]
\begin{center}
\includegraphics{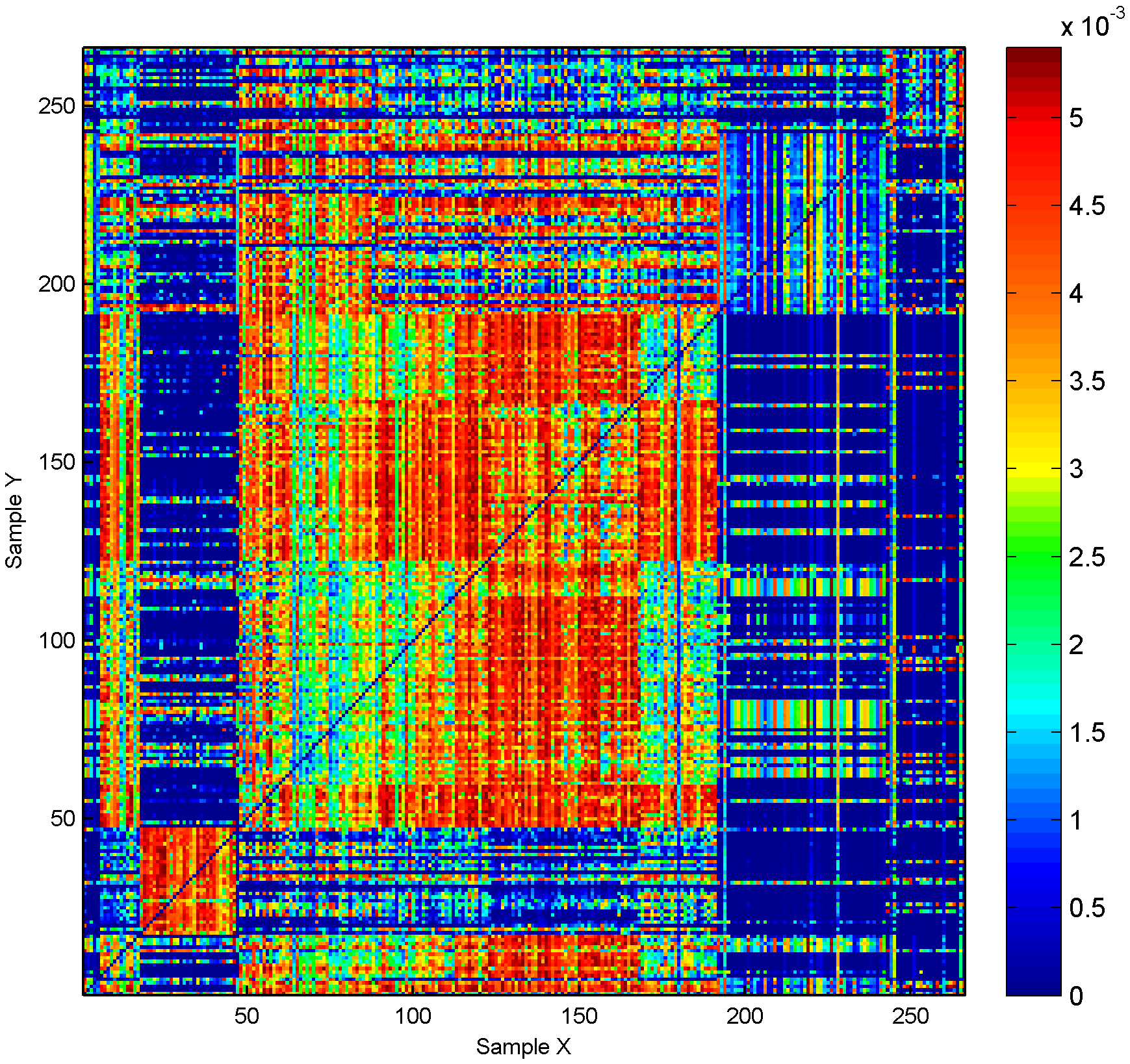}
\end{center}
\caption{{\bf Error estimates.} Heat map of $S(n_y)$ sorted by site location metadata. Here, the $x$-axis also denotes the sample from the environment corresponding to urn-$x$, and similarly for the $y$-axis, and the entries on the diagonal are again set to zero.}
\label{Fig:ThetaVarMap}
\end{figure}

\begin{figure}[ht!]
\begin{center}
\includegraphics{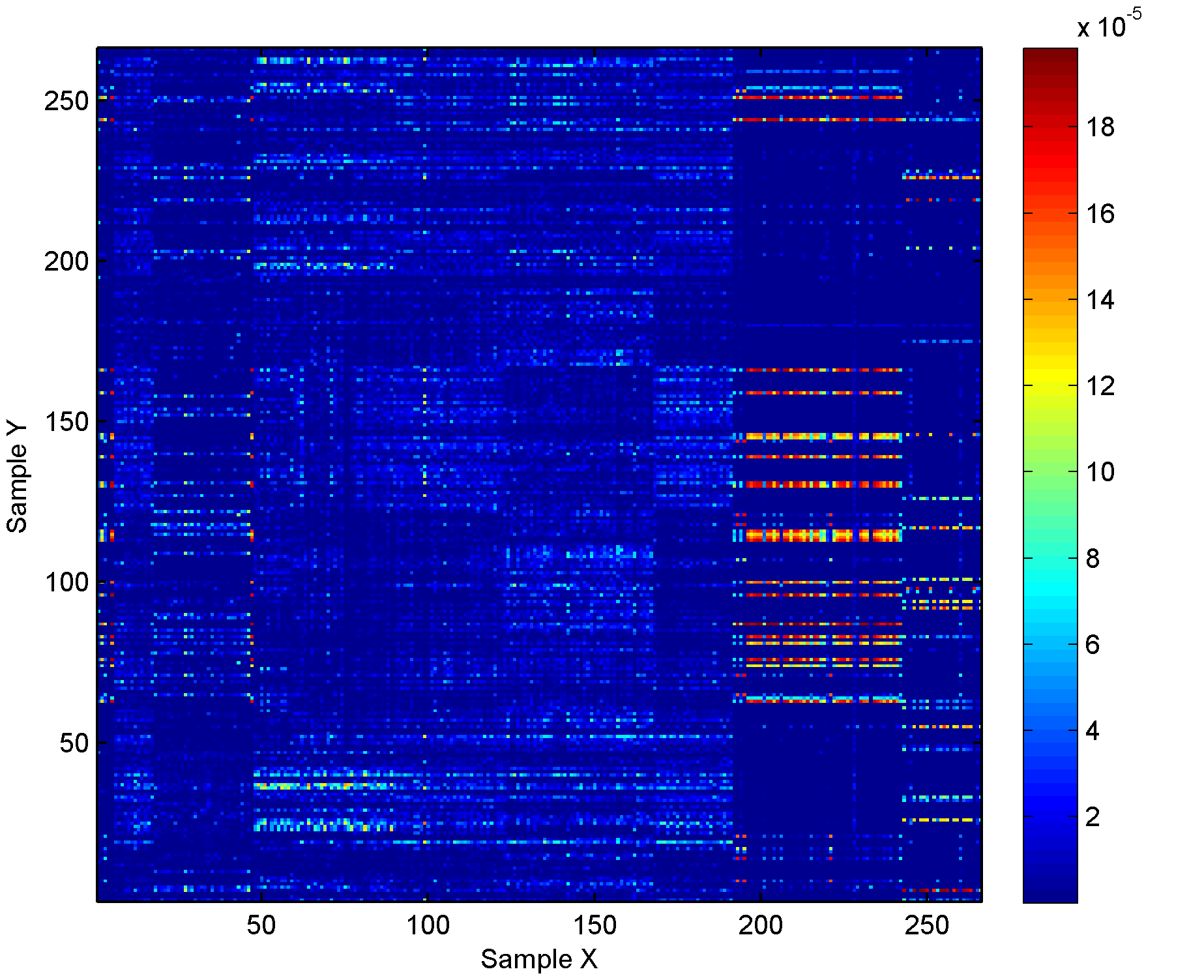}
\end{center}
\caption{{\bf Discrete derivative estimates.} Heat map of $|\hat{\theta}(n_y)-\hat{\theta}(n_y-1)|$, sorted by site location metadata, following the same conventions as in the previous figures.}
\label{Fig:DThetaMap}
\end{figure}

\begin{figure}[ht!]
\begin{center}
\includegraphics{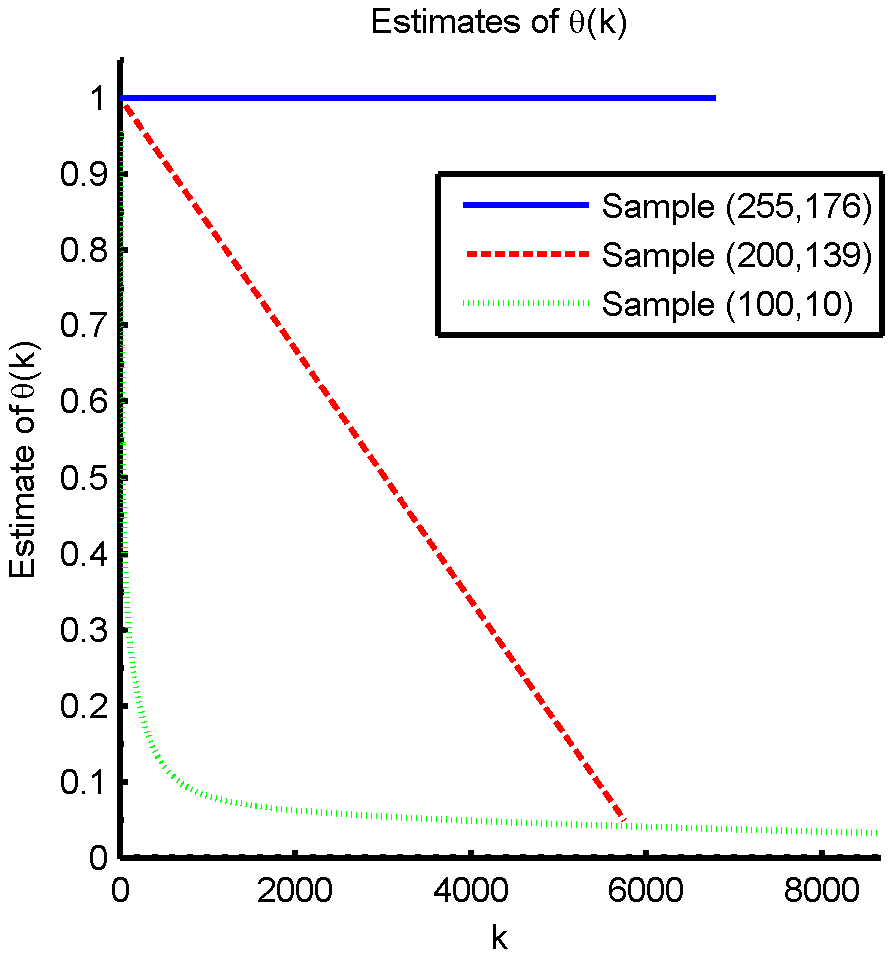}
\end{center}
\caption{{\bf Sequential estimation.} Plots of $\hat\theta(k)$, with $k=1:n_y$, for three pairs of samples of the HMP data.}
\label{Fig:ThetaCurveMap}
\end{figure}

\section*{Tables}

\begin{table}[h!]
\caption{{\bf HMP data.} Summary of V35 16S rRNA data processed by Qiime into an OTU table.}
\begin{center}
\begin{tabular}{|l|l|l|}
     \hline
Body Supersite &Body Subsite &Assigned Labels\\\hline
Airways&Anterior Nares&1-5  \\\cline{2-3}
			 &Throat        &6-17 \\\hline
     	 
Gastrointestinal Tract&Stool&18-47\\\hline
     					        
Oral&Attached/Keratinized Gingiva&48-59  \\\cline{2-3}
		&Buccal Mucosa               &60-76  \\\cline{2-3}
		&Hard Palate                 &77-90  \\\cline{2-3}
		&Palatine Tonsils            &91-112 \\\cline{2-3}
		&Saliva											 &113-122\\\cline{2-3}
		&Subgingival Plaque          &123-144\\\cline{2-3}
		&Supragingival Plaque        &145-167\\\cline{2-3}
		&Tongue Dorsum               &168-191\\\hline

Skin&Left Antecubital Fossa     &192-195\\\cline{2-3}
    &Left Retroauricular Crease &196-217\\\cline{2-3}
    &Right Antecubital Fossa    &218-222\\\cline{2-3}
    &Right Retroauricular Crease&223-242\\\hline

Urogenital Tract&Mid Vagina       &243-248\\\cline{2-3}
							  &Posterior Fornix &249-259\\\cline{2-3}
							  &Vaginal Introitus&260-266\\\hline
\end{tabular}
\end{center}
\label{tab:1}
\end{table}

\begin{table}[ht!]
\caption{{\bf Sample comparisons.} Summary of estimates for three pairs of samples of the HMP data.}
\begin{center}
\begin{small}
\begin{tabular}{|c|c|c|c|c|c|c|c|c|}
\hline
Urn-$x$ & Urn-$y$ & $n_x$ & $n_y$ & $\hat{\theta}(n_y)$ & $\hat\theta(n_y)-\hat\theta(n_y-1)$ & Regression Error & $S(n_y)$ & $\hat\rho^{n_y}$ \\
\hline
255 & 176 & 5054  & 6782 & 0.9998 & 0.0                      &0.0 & 1.9892$\times10^{-4}$ & 0.0     \\
\hline
200 & 139 & 12747 & 5739 & 0.0499 & -1.6533$\times 10^{-4}$ & 6.8306$\times10^{-6}$ & 1.9286$\times10^{-3}$ & 0.9997 \\
\hline
100 & 10  & 6206  & 8655 & 0.0324 & -2.9416$\times 10^{-6}$ & 0.0438 & 2.2477$\times10^{-3}$ & 0.5130 
\\
\hline
\end{tabular}
\end{small}
\end{center}
\label{tab:2}
\end{table}

\newpage

\section*{Supporting Information Legends}

\noindent{\bf File S1.} Summary Metadata related to Table 1 (tab-limited text file).\\

\noindent{\bf File S2.} OTU table related to Table 2 and Figures 1-4 (tab-limited text file).

\end{document}